\documentclass[print,wide,eng]{draft}
\hypersetup{hidelinks,pdftitle={A Transfer Principle for Elliptic Carleman Estimates: Application to Discrete Calder\'on Stability}}
\title{A Transfer Principle for Elliptic Carleman Estimates: Application to Discrete Calder\'on Stability}
\author{Qi L\"u\thanks{School of Mathematics, Sichuan University, Chengdu, 610064, China. {\small\it E-mail:} {\small\tt lu@scu.edu.cn}. Qi L\"u was supported by the National Natural Science Foundation of China
		under grant 12625102. } \quad and \quad Yu Wang\thanks{School of Mathematics, Southwest Jiaotong University, Chengdu, 611756, China. {\small\it E-mail:} {\small\tt yuwangmath@163.com.} Yu Wang was supported by the National Natural Science
		Foundation of China under grant 12401589 and the Sichuan Science and
		Technology Program under grant 2026NSFSC0777.}}
\date{}
\begin{document}
\maketitle
\begin{abstract}
We develop a transfer principle that lifts continuous elliptic Carleman estimates to mesh-uniform estimates for conforming Galerkin discretizations. The main difficulties are twofold: discrete states generally lack the \(H^2\)-regularity required by the continuous estimate, and the Carleman parameter \(s\) must be controlled relative to the mesh size \(h\). We overcome both by introducing a shifted continuous problem whose Galerkin solution is exactly the conforming reconstruction of a prescribed discrete state. Weighted approximation then transfers the inequality back to the original discrete operator without adding stabilization, and yields the range \(s\le ch^{-4/5}\) for radial weights. We verify the framework for piecewise affine finite elements with consistent or lumped mass on fitted, uniformly shape-regular meshes, and for standard Cartesian finite differences through their finite-element realization. As an application, we study the discrete Calder\'on problem for bounded complex Schr\"odinger potentials in dimension \(d\ge3\). Assuming agreement in a fixed connected boundary neighborhood and uniform stability of the discrete Dirichlet problems, we obtain stability from Dirichlet inputs and variational flux observations on independently prescribed, possibly disjoint boundary patches. The resulting estimate in \(H^{-r}\), \(r>d/2\), retains the single logarithmic data term and the \(O(h^{1/2})\) mesh remainder of the full-data estimate.
\end{abstract}



\section{Introduction}\label{sec:intro}

Many elliptic inverse and unique continuation problems ask how accurately a coefficient or a solution can be determined from observations on part of the domain or its boundary. Carleman estimates address this question through weighted inequalities that control a function by its equation residual and observations. By varying a large parameter in the weight, these inequalities yield quantitative propagation of smallness and stability estimates \cite{LR2012,ARRV2009}. For a numerical discretization, however, such stability must persist with constants that remain bounded as the mesh is refined. Forward convergence for smooth solutions does not by itself imply this mesh-uniform stability of the discrete inverse problem.

We establish a transfer principle that lifts continuous elliptic Carleman estimates to conforming Galerkin discretizations through weighted approximation. The transferred inequality holds for every discrete state and uses the original discrete operator, without adding stabilization. We verify the framework for \(P_1\) finite elements with consistent or lumped mass and for standard Cartesian finite differences through their finite-element realization. Our main application is the discrete Calder\'on problem: under boundary agreement and uniform Dirichlet stability, we obtain stability with Dirichlet inputs supported on one boundary patch and fluxes observed on another, independently prescribed and possibly disjoint patch.

While the primary focus of this work is the rigorous mathematical foundation of the transfer principle and its analytical consequences, the explicit parameter ranges (such as \(s \le ch^{-4/5}\)) and the exact discrete matrix structures provided (e.g., Schur complements and mass matrices) are designed to be directly implementable. We note that numerical experiments for the discrete Calderón problem are highly nontrivial due to the extreme sensitivity of the inverse problem and the high-frequency nature of the CGO probes, which can lead to severe conditioning issues in floating-point arithmetic. Consequently, a comprehensive numerical study is beyond the scope of this paper and is left as an interesting direction for future work.

\subsection{A transfer principle for conforming discretizations}

A direct application of a continuous Carleman estimate to a finite-element function encounters two fundamental obstructions. First, conforming piecewise affine functions generally have jumps in their gradients across element interfaces and fail to belong to \(H^2(\Omega)\), the regularity required by the continuous estimate. Second, differentiating the exponential weight \(\theta=e^{s\rho}\) introduces powers of the large Carleman parameter \(s\), so an approximation error that is small as the mesh size \(h\) tends to zero for fixed \(s\) may cease to be small when \(s\) increases. The transfer must therefore specify a relation between \(h\) and \(s\) under which the reconstruction errors can be absorbed by the weighted terms controlled in the Carleman estimate; related parameter restrictions occur in \cite[Section~6.1]{Zuazua2007} and in direct discrete estimates \cite{Boyer2010,Boyer2010a,EdG}.

We address both difficulties simultaneously through a shifted continuous problem. For example, let \(V_h\subset H_0^1(\Omega)\) be a conforming piecewise affine space on a bounded convex polytope \(\Omega\), and let \(A_h\) be the consistent-mass Galerkin operator for \(-\Delta\). Given \(z_h\in V_h\), choose \(\kappa>0\) sufficiently large, independently of \(h,s\), and solve
\begin{align*}
(-\Delta+\kappa s^2)U=A_hz_h+\kappa s^2z_h,\qquad U|_{\partial\Omega}=0.
\end{align*}
The Galerkin solution is exactly \(z_h\), while \(U\in H^2(\Omega)\) is regular enough for the continuous estimate. The positive shift controls the terms generated by differentiating the weight. Weighted energy estimates and a duality argument with the reciprocal weight \(\theta^{-1}\) control \(U-z_h\) and allow the comparison errors to be absorbed. This construction adapts elliptic reconstruction \cite[Definition~2.1]{MN2003} by adding a shift depending on \(s\).

As a representative consequence, consider fitted, uniformly shape-regular simplicial meshes and fix an observation ball \(\omega\Subset\Omega\) centered at \(x_*\). For \(\rho(x)=e^{-\lambda|x-x_*|^2}\) with \(\lambda\) fixed sufficiently large, \cref{thm:radial-discrete} gives
\begin{align}\label{intro:discrete}
\lVert \theta\nabla z_h\rVert_{L^2(\Omega)}^2+s^2\lVert \theta z_h\rVert_{L^2(\Omega)}^2
\leq C\bigl(s^{-1}\lVert \theta A_hz_h\rVert_{L^2(\Omega)}^2
+s^2\lVert \theta z_h\rVert_{L^2(\omega)}^2\bigr)
\end{align}
for every \(z_h\in V_h\), \(0<h\leq h_0\), and \(s_0\leq s\leq ch^{-4/5}\), with constants independent of \(h,s\). We verify the numerical hypotheses for both consistent and lumped mass on fitted, conforming, uniformly shape-regular meshes; see \cref{prop:transfer-realizations}. No quasi-uniformity or additional quantitative grading condition is required, so local refinement is allowed. The abstract result, \cref{thm:transfer}, also covers non-self-adjoint operators under separate weighted regularity and approximation assumptions for the primal and adjoint problems, together with the stated compatibility conditions. The shifted reconstruction is shared with our related parabolic work \cite[Section~1.3]{ParabolicTransfer}, while the present work develops the elliptic realizations and the boundary continuation and inverse stability results below.

\subsection{Application to the discrete Calder\'on problem}
We consider the inverse problem for a conforming discretization of the Schr\"odinger equation \((-\Delta+q)u=0\). Given discrete Dirichlet values on the boundary mesh, we solve the finite-element system and measure the resulting variational boundary flux. The discrete Dirichlet-to-Neumann map \(\Lambda_h(q)\) sends the prescribed values to this flux; in nodal coordinates, it is the Schur complement obtained after eliminating the interior degrees of freedom. The discrete Calder\'on problem asks how accurately these boundary responses determine the potential used in the discrete operator.

In the continuous theory, complex geometrical optics (CGO) solutions, which probe Fourier modes of the coefficient, yield uniqueness and stability through boundary identities \cite{SU1987,Alessandrini1988}; Carleman estimates also give partial-data uniqueness \cite{KSU2007}. Ervedoza and de Gournay \cite{EdG} constructed exact discrete CGO solutions and proved mesh-uniform full-data stability, including a lumped-mass finite-element realization on structured meshes. On Cartesian grids of a cube, Lecaros, Ortega, P\'erez, and de Teresa \cite{Lecaros2023a} restricted the flux observation to all faces except one, and Zhao and Yuan \cite[Theorem~2.9]{ZY} reduced the observation region to any fixed nonempty open subset of a single face. Both partial-data results still allow Dirichlet inputs on the whole boundary; restricting the input as well is a question raised in \cite[Remark~2.13]{ZY}.

We obtain stability with Dirichlet inputs supported on a patch \(\Gamma_D\) and fluxes observed on another patch \(\Gamma_N\), independently prescribed and possibly disjoint. This provides the input localization proposed in \cite[Remark~2.13]{ZY} for the Laplace-based schemes considered here, under the boundary agreement and uniform Dirichlet stability assumptions below. Moreover, localizing both input and observation preserves the single logarithmic modulus and \(h^{1/2}\) mesh remainder of our full-data estimate. The numerical setting also extends beyond the Cartesian grids of \cite{Lecaros2023a,ZY}: we treat both consistent and lumped mass on fitted, conforming, uniformly shape-regular meshes, allowing local refinement.

Specifically, let \(\Omega\subset\mathbb R^d\), \(d\geq3\), be a bounded connected Lipschitz polyhedral domain, with the fitted mesh geometry of \cref{sec:models}. Assume that the bounded complex potentials \(q_1,q_2\) agree in a fixed connected neighborhood of the entire boundary, with the nodal interpretation for lumped mass specified in \cref{thm:patch}. We also assume that their discrete Dirichlet problems satisfy the uniform stability bound \cref{eq:dir-infsup}. The patches \(\Gamma_D,\Gamma_N\) are fixed nonempty open subsets of the interiors of boundary faces. Let \(\delta_{\mathrm{loc}}\) be the operator norm of the difference of the two boundary maps for inputs supported on \(\Gamma_D\) and fluxes observed on \(\Gamma_N\), as defined in \cref{eq:two-patch-data}. To compare potentials as functions, \(\mathcal Q_h\) retains the original potential for consistent mass and represents nodal values as constants on their control volumes for lumped mass, with zero extension outside \(\Omega\). For \(r>d/2\), \cref{thm:patch} yields
\begin{align}\label{intro:inverse}
\lVert \mathcal Q_h(q_1-q_2)\rVert_{H^{-r}(\mathbb R^d)}
\leq C\bigl(h^{1/2}+[\log(2+\delta_{\mathrm{loc}}^{-1})]^{-1}\bigr),
\qquad 0<h\leq h_0.
\end{align}
Here \(C,h_0\) may depend on the fixed geometry and patches but are independent of the mesh and the particular potentials within the prescribed bounds; the logarithmic term is zero when \(\delta_{\mathrm{loc}}=0\). The \(h^{1/2}\) scale in the underlying full-data Fourier estimate is already present in \cite[Theorem~2.7]{EdG}.

We reduce the two-patch problem to the full-data problem by two quantitative continuation steps, separated by reciprocity. With the input supported on \(\Gamma_D\), boundary continuation first controls the flux difference on the whole boundary from its restriction to \(\Gamma_N\). Reciprocity interchanges the roles of input and observation, so this also controls the flux difference on \(\Gamma_D\) for full-boundary inputs. A second continuation then controls the full-boundary flux difference for these inputs, giving the full-data discrepancy. Agreement of the potentials near the boundary makes the difference equation homogeneous where continuation is applied. In the continuous setting, boundary agreement is used in \cite{AK2012} to compare local and interior full-data maps, while quantitative Runge approximation provides a related approach \cite{RS2019}.

The full-data estimate and boundary continuation use different Carleman weights: a linear function \(\rho\) yields exact discrete CGO solutions, while the radial choice above yields the continuation estimates; related discrete propagation results appear in \cite{GM2014,LM2015,FBRRS}. The CGO construction retains the original stiffness and potential forms; both Carleman parameter ranges in \cref{cor:budgets} permit \(s\) to reach the scale \(h^{-1/2}\) used in the inverse estimate. The continuation steps bound the full-data discrepancy by a positive power of \(\delta_{\mathrm{loc}}\), up to a remainder of order \(e^{-ch^{-4/5}}\); see \cref{eq:boundary-cont}. In the full-data logarithmic estimate, this remainder contributes \(O(h^{4/5})\), which is bounded by the existing \(h^{1/2}\) term, while the power of \(\delta_{\mathrm{loc}}\) introduces no additional logarithm. As a further consequence, \cref{cor:finite-dimensional} gives mesh-uniform Lipschitz stability on sufficiently fine meshes for coefficient differences in a finite-dimensional model fixed independently of the mesh, under the same two-patch hypotheses.

\Cref{sec:interfaces,sec:realizations} establish the transfer principle and verify its numerical hypotheses. \Cref{sec:weights,sec:cgo,sec:local} develop the weighted estimates, discrete CGO solutions, and quantitative continuation. \Cref{sec:inverse} applies these tools to full-data and two-patch stability.

\section{A transfer principle for elliptic Carleman estimates}\label{sec:interfaces}
Our goal is to transfer a continuous Carleman estimate to the discrete operator \(A_h\). The main obstacle is that discrete states lack the regularity required by the continuous estimate. To overcome this, we construct, for each discrete state, a shifted continuous problem whose Galerkin solution is exactly its conforming reconstruction. Weighted energy and \(L^2\) comparison estimates then allow us to apply the continuous Carleman estimate to the reconstruction and absorb the error. In this section, we establish this mechanism in an abstract setting. We first state the continuous assumptions, then the discrete numerical assumptions, and finally prove the transfer theorem and its parameter budget.

\subsection{Continuous setting and analytic assumptions}

We begin by fixing the continuous setting. Let \(\Omega\subset\mathbb{R}^d\) be a bounded Lipschitz domain and set \(H=L^2(\Omega;\mathbb{C})\). Inner products and sesquilinear forms are linear in the first argument. Let \(V\) be a closed subspace of \(H^1(\Omega;\mathbb{C})\), dense in \(H\), and let \(a:V\times V\to\mathbb{C}\) be a continuous sesquilinear form, not necessarily Hermitian. Let \(A\) be a densely defined closed operator in \(H\), with \(D(A)\subset V\), associated with \(a\) through
\begin{align*}
a(u,v)=(Au,v)_H,\qquad u\in D(A),\quad v\in V,
\end{align*}
and let \(A^*\) be its Hilbert-space adjoint. We assume \(D(A^*)\subset V\) and the adjoint identity
\begin{align}\label{eq:adjoint-form}
a(v,w)=(v,A^*w)_H,\qquad v\in V,\quad w\in D(A^*).
\end{align}
Fix a family of positive weights \(\theta=\theta_s\), \(s\geq s_0\geq1\), such that \(\theta,\theta^{-1}\in W^{1,\infty}(\Omega)\) for each \(s\) and \(|\nabla\log\theta|\leq Cs\) with a constant independent of \(s\).
For \(\sigma=\theta\) or \(\theta^{-1}\) and a measurable \(G\subset\Omega\), define the weighted energy norm
\begin{align*}
\mathcal{N}_{\sigma,G}(v)^2
=\lVert \sigma\nabla v\rVert_{L^2(G)}^2+s^2\lVert \sigma v\rVert_{L^2(G)}^2,
\end{align*}
and the corresponding second-order quantity
\begin{align*}
\mathcal{R}_{\sigma,G}(v)
=\lVert \sigma D^2v\rVert_{L^2(G)}
+s\lVert \sigma\nabla v\rVert_{L^2(G)}+s^2\lVert \sigma v\rVert_{L^2(G)}.
\end{align*}
Here \(D^2v\) denotes the Hessian with its Euclidean component norm.

Unless otherwise stated, \(C,c>0\) denote generic constants, which may change from one occurrence to the next. They are independent of \(h,s\) and of the functions and loads being estimated. We write \(X\asymp Y\) when \(cY\leq X\leq CY\) with such constants.

The shifted problem requires a sufficiently large parameter \(\kappa>0\). We impose the following assumptions uniformly for \(s\geq s_0\).
\begin{enumerate}[label=\textbf{(A\arabic*)},ref=(A\arabic*)]
\item\label[assumption]{ass:weighted} \emph{Weighted form bounds.}
Multiplication by \(\theta^2\) preserves \(V\). For some fixed \(\kappa>0\), the shifted form
\begin{align*}
a_s(u,v)=a(u,v)+\kappa s^2(u,v)_H
\end{align*}
satisfies the continuity bound
\begin{align}\label{eq:form-bounds}
|a_s(u,v)|\leq C\mathcal{N}_{\theta,\Omega}(u)
\mathcal{N}_{\theta^{-1},\Omega}(v),\qquad u,v\in V,
\end{align}
and the weighted coercivity estimate
\begin{align}\label{eq:form-coercivity}
\operatorname{Re} a_s(v,\theta^2v)\geq c\mathcal{N}_{\theta,\Omega}(v)^2,
\qquad v\in V.
\end{align}
\item\label[assumption]{ass:regularity} \emph{Primal and adjoint shifted regularity.}
We assume $D(A),D(A^*)\subset H^2(\Omega)$. For the same fixed \(\kappa\) and every \(g\in H\), the equation \((A+\kappa s^2)U=g\) has a unique solution \(U\in D(A)\) satisfying
\begin{align}\label{eq:screen-reg}
\mathcal{R}_{\theta,\Omega}(U)\leq C\lVert \theta g\rVert_{L^2(\Omega)}.
\end{align}
For every \(f\in H\), the adjoint equation \((A^*+\kappa s^2)Y=f\) has a unique solution \(Y\in D(A^*)\) satisfying
\begin{align}\label{eq:screen-reg-adjoint}
\mathcal{R}_{\theta^{-1},\Omega}(Y)\leq C\lVert \theta^{-1}f\rVert_{L^2(\Omega)}.
\end{align}
\item\label[assumption]{ass:carleman} \emph{Continuous Carleman estimate.}
For a fixed \(\gamma\in[0,1]\), a nonnegative quadratic form \(\mathcal{E}_s\) on \(V\), and a seminorm \(\mathcal{O}_s\) on \(V\),
\begin{align}\label{eq:continuousabstract}
\mathcal{E}_s(v)\leq C\bigl(s^{-\gamma}\lVert \theta Av\rVert_{L^2(\Omega)}^2
+\mathcal{O}_s(v)^2\bigr),
\qquad v\in D(A),
\end{align}
where \(\mathcal{E}_s(v)\asymp\mathcal{N}_{\theta,\Omega}(v)^2\) and \(\mathcal{O}_s(v)\leq C\mathcal{N}_{\theta,\Omega}(v)\) for \(v\in V\).
\end{enumerate}
We fix this choice of \(\kappa\) before choosing any numerical smallness threshold.

\subsection{Discrete scheme and numerical assumptions}
We now introduce the discrete counterparts. For each \(h\), let \(X_h\) be a finite-dimensional complex Hilbert space with prescribed inner product \((\cdot,\cdot)_{X_h}\). Let \(R_h:X_h\to V\) be an injective linear conforming reconstruction with range \(V_h\). Define the discrete operator \(A_h\) by
\begin{align}\label{eq:stiffnessinterface}
(A_hz,w)_{X_h}=a(R_hz,R_hw),\qquad z,w\in X_h.
\end{align}
For conforming finite elements, one may take \(X_h=V_h\) and \(R_h\) to be the natural inclusion, while the inner product on \(X_h\) may represent either consistent or lumped mass.

The numerical assumptions below require that the discrete space can adequately represent the continuous solution at the shifted frequency, and that weighted testing is possible within \(V_h\). They are required uniformly for the values of \(h>0\) and \(s\geq s_0\) under consideration.
\begin{enumerate}[label=\textbf{(N\arabic*)},ref=(N\arabic*)]
\item\label[assumption]{ass:loads} \emph{Stable load lifting.}
Let \(\lVert\cdot\rVert_{X_h,\theta}\) be a weighted norm on \(X_h\) used to measure discrete loads. There is a linear lift \(\mathcal{L}_h:X_h\to H\) satisfying
\begin{align}\label{eq:load-lifting}
(\mathcal{L}_h f,R_hz)_H=(f,z)_{X_h},\qquad f,z\in X_h,
\end{align}
and the weighted bound
\begin{align}\label{eq:load-stability}
\lVert \theta\mathcal{L}_h f\rVert_{L^2(\Omega)}\leq C\lVert f\rVert_{X_h,\theta}.
\end{align}
\item\label[assumption]{Papprox} \emph{Primal and adjoint weighted approximation.}
Let \(\alpha(h,s),\alpha^*(h,s)\geq0\) be the primal and adjoint approximation bounds, respectively. The primal domain satisfies
\begin{align}\label{eq:approximation}
\inf_{v_h\in V_h}\mathcal{N}_{\theta,\Omega}(v-v_h)
\leq \alpha(h,s)\mathcal{R}_{\theta,\Omega}(v),
\qquad v\in D(A).
\end{align}
The adjoint domain satisfies
\begin{align}\label{eq:approximation-adjoint}
\inf_{w_h\in V_h}\mathcal{N}_{\theta^{-1},\Omega}(w-w_h)
\leq \alpha^*(h,s)\mathcal{R}_{\theta^{-1},\Omega}(w),
\qquad w\in D(A^*).
\end{align}
\item\label[assumption]{Ptest} \emph{Weighted testing.}
For every \(v_h\in V_h\), there is a test function \(w_h\in V_h\) such that
\begin{align}\label{eq:testsuper}
\mathcal{N}_{\theta^{-1},\Omega}
(w_h-\theta^2v_h)
\leq\frac{c}{2C}\mathcal{N}_{\theta,\Omega}(v_h),
\end{align}
where \(C\) and \(c\) are the constants in \cref{eq:form-bounds,eq:form-coercivity}, respectively.
\item\label[assumption]{ass:observation} \emph{Energy and observation compatibility.}
There are a nonnegative quadratic form \(\mathcal{E}_{s,h}\) and a seminorm \(\mathcal{O}_{s,h}\) on \(X_h\) representing the discrete energy and observation, with
\begin{align}\label{eq:energy-compatibility}
\mathcal{E}_{s,h}(z)\asymp\mathcal{E}_s(R_hz),\qquad z\in X_h,
\end{align}
and
\begin{align}\label{eq:observation-compatibility}
\mathcal{O}_s(R_hz)\leq C\mathcal{O}_{s,h}(z),\qquad z\in X_h.
\end{align}
\end{enumerate}

\subsection{The transfer theorem}

The transfer theorem first bounds the discrete energy in terms of a load \(F\in H\) representing \(a(R_hz,\cdot)\) on \(V_h\). Under \cref{ass:loads}, choosing \(F=\mathcal{L}_h A_hz\) yields an estimate in terms of the discrete residual \(A_hz\). The key condition is that the comparison errors, measured by \(\alpha\) and \(\alpha^*\), are small enough to be absorbed by the coercive terms.

\begin{theorem}\label{thm:transfer}
	Assume \cref{ass:weighted,ass:regularity,ass:carleman} and \cref{Papprox,Ptest,ass:observation} hold. There exist constants \(c,C>0\), depending only on the fixed data, such that for all \(h>0\) and \(s\geq s_0\) satisfying
	\begin{align}\label{eq:transfer-criterion}
		s^2\alpha(h,s)^2+s^{6-\gamma}\alpha(h,s)^2\alpha^*(h,s)^2\leq c,
	\end{align}
	the following estimates hold. For every \(z\in X_h\) and \(F\in H\) satisfying
	\begin{align}\label{eq:load-identity}
		a(R_hz,R_hw)=(F,R_hw)_H,\qquad w\in X_h,
	\end{align}
	we have
	\begin{align}\label{eq:exact-load-CE}
		\mathcal{E}_{s,h}(z)
		\leq C\bigl(s^{-\gamma}\lVert \theta F\rVert_{L^2(\Omega)}^2
		+\mathcal{O}_{s,h}(z)^2\bigr).
	\end{align}
	If, in addition, \cref{ass:loads} holds, then
	\begin{align}\label{eq:discrete-CE}
		\mathcal{E}_{s,h}(z)
		\leq C\bigl(s^{-\gamma}\lVert A_hz\rVert_{X_h,\theta}^2
		+\mathcal{O}_{s,h}(z)^2\bigr),
		\qquad z\in X_h.
	\end{align}
\end{theorem}

\Cref{eq:exact-load-CE} is useful after localization: a cutoff commutator can be represented by a local \(L^2\) load even when the strong discrete operator contains the inverse of a global mass matrix.

For piecewise affine approximation, the two error scales yield an explicit range of the Carleman parameter.
\begin{corollary}\label{cor:budgets}
	Assume the analytic and numerical assumptions of \cref{thm:transfer} hold uniformly for \(h>0\), \(s\geq s_0\), and \(hs\leq c_0\) (with \(c_0>0\) fixed), and that \(\alpha(h,s),\alpha^*(h,s)\leq Ch\). Then there exist constants \(c,h_0>0\), depending only on the fixed data, such that the estimates of \cref{thm:transfer} hold whenever
	\begin{align}\label{eq:p1-range}
		0<h\leq h_0,\qquad s_0\leq s\leq ch^{-4/(6-\gamma)}.
	\end{align}
\end{corollary}
The exponents in \cref{eq:p1-range} are \(2/3\) for \(\gamma=0\) and \(4/5\) for \(\gamma=1\).
\begin{proof}
The left-hand side of \cref{eq:transfer-criterion} is bounded by \(C(h^2s^2+h^4s^{6-\gamma})\).
For \(h\leq1\) and \(s\leq ch^{-4/(6-\gamma)}\), the two terms are bounded by \(Cc^2\) and \(Cc^{6-\gamma}\), respectively, while \(hs\leq c\).
Choose \(c\leq c_0\) so that \cref{eq:transfer-criterion} holds, and then \(h_0\leq1\) so that the interval in \cref{eq:p1-range} is nonempty.
\end{proof}

\subsection{Weighted Galerkin comparison and reciprocal-weight duality}\label{sec:error}

The proof of \cref{thm:transfer} relies on comparing the exact continuous solution of the shifted problem with its Galerkin approximation. The next lemma combines a weighted C\'ea estimate for the energy error with an Aubin--Nitsche argument for the reciprocal weight to control the \(L^2\) error.
\begin{lemma}\label{lem:comparison}
	Assume \cref{ass:weighted,ass:regularity} and \cref{Papprox,Ptest} hold. Let \(g\in H\) and \(U=(A+\kappa s^2)^{-1}g\). There is a unique \(U_h\in V_h\) satisfying
	\begin{align}\label{eq:screened-galerkin}
		a_s(U_h,v_h)=(g,v_h)_H,\qquad v_h\in V_h.
	\end{align}
	The energy error satisfies
	\begin{align}\label{eq:comparison}
		\mathcal{N}_{\theta,\Omega}(U-U_h)
		\leq C\alpha(h,s)\lVert \theta g\rVert_{L^2(\Omega)}.
	\end{align}
	Moreover, the weighted \(L^2\) error satisfies
	\begin{align}\label{eq:comparison-l2}
		\lVert \theta(U-U_h)\rVert_{L^2(\Omega)}
		\leq C\alpha(h,s)\alpha^*(h,s)\lVert \theta g\rVert_{L^2(\Omega)}.
	\end{align}
\end{lemma}
\begin{proof}
	\emph{Step 1: discrete stability and the energy error.}
	We first establish discrete stability. For any \(\xi_h\in V_h\), choose a test function \(w_h\in V_h\) as in \cref{eq:testsuper}. By \cref{eq:form-bounds,eq:form-coercivity,eq:testsuper}, we have
	\begin{align*}
		\operatorname{Re} a_s(\xi_h,w_h)
		\geq\frac{c}{2}\mathcal{N}_{\theta,\Omega}(\xi_h)^2.
	\end{align*}
	Thus, if \(\xi_h\) is a homogeneous solution of \cref{eq:screened-galerkin}, it must vanish, which implies the discrete system is uniquely solvable.
	
	Next, we estimate the energy error. Fix an arbitrary approximant \(p_h\in V_h\) and set \(\eta=U-p_h\) and \(\xi_h=U_h-p_h\). Galerkin orthogonality gives \(a_s(\xi_h,v_h)=a_s(\eta,v_h)\) for all \(v_h\in V_h\). 
	To bound \(\xi_h\), choose \(w_h\in V_h\) as above. The product rule and the bound \(|\nabla\log\theta|\leq Cs\) imply \(\mathcal{N}_{\theta^{-1},\Omega}(\theta^2\xi_h)\leq C\mathcal{N}_{\theta,\Omega}(\xi_h)\). Together with \cref{eq:testsuper}, this yields
	\begin{align*}
		\mathcal{N}_{\theta^{-1},\Omega}(w_h) \leq C\mathcal{N}_{\theta,\Omega}(\xi_h).
	\end{align*}
	Testing the Galerkin orthogonality with \(w_h\), we obtain
	\begin{align*}
		\frac{c}{2}\mathcal{N}_{\theta,\Omega}(\xi_h)^2
		\leq\operatorname{Re}a_s(\xi_h,w_h)
		=\operatorname{Re}a_s(\eta,w_h)
		\leq C\mathcal{N}_{\theta,\Omega}(\eta)\mathcal{N}_{\theta,\Omega}(\xi_h).
	\end{align*}
	Dividing by \(\mathcal{N}_{\theta,\Omega}(\xi_h)\) yields \(\mathcal{N}_{\theta,\Omega}(\xi_h)\leq C\mathcal{N}_{\theta,\Omega}(\eta)\). The triangle inequality then gives
	\begin{align*}
		\mathcal{N}_{\theta,\Omega}(U-U_h) \leq \mathcal{N}_{\theta,\Omega}(\eta) + \mathcal{N}_{\theta,\Omega}(\xi_h) \leq C\mathcal{N}_{\theta,\Omega}(\eta).
	\end{align*}
	Since \(p_h\in V_h\) was arbitrary, we may take the infimum over \(p_h\). Applying \cref{eq:approximation,eq:screen-reg} then proves \cref{eq:comparison}.
	
	\emph{Step 2: the reciprocal-weight adjoint problem.}
	To obtain the \(L^2\) estimate, set \(e=U-U_h\) and consider the adjoint problem \((A^*+\kappa s^2)Y=\theta^2e\). 
	By the adjoint regularity \cref{eq:screen-reg-adjoint}, we have
	\begin{align*}
		\mathcal{R}_{\theta^{-1},\Omega}(Y)\leq C\lVert \theta e\rVert_{L^2(\Omega)}.
	\end{align*}
	Using \cref{eq:adjoint-form,eq:form-bounds} and Galerkin orthogonality, for any \(y_h\in V_h\) we have
	\begin{align*}
		\lVert \theta e\rVert_{L^2(\Omega)}^2
		&=(e,(A^*+\kappa s^2)Y)_H =a_s(e,Y) =a_s(e,Y-y_h)  \leq C\mathcal{N}_{\theta,\Omega}(e) \mathcal{N}_{\theta^{-1},\Omega}(Y-y_h).
	\end{align*}
	Taking the infimum over \(y_h\in V_h\) and applying the adjoint approximation \cref{eq:approximation-adjoint}, we obtain
	\begin{align*}
		\lVert \theta e\rVert_{L^2(\Omega)}^2
		\leq C\alpha^*(h,s)\mathcal{N}_{\theta,\Omega}(e) \lVert \theta e\rVert_{L^2(\Omega)}.
	\end{align*}
	Dividing by \(\lVert \theta e\rVert_{L^2(\Omega)}\) and combining the result with \cref{eq:comparison} proves \cref{eq:comparison-l2}.
\end{proof}

\subsection{Proof of the transfer theorem}\label{sec:transfer}
With the comparison lemma in hand, we can now prove the main transfer theorem by applying the continuous Carleman estimate to the shifted reconstruction and absorbing the error.

\begin{proof}[Proof of \cref{thm:transfer}]
	\emph{Step 1: reconstruct the prescribed discrete state.}
	Let $z\in X_h$ and \(F\in H\) satisfy \cref{eq:load-identity}. Define the auxiliary load
	\begin{align}\label{eq:aux-load}
		g=F+\kappa s^2R_hz,
	\end{align}
	and let \(U=(A+\kappa s^2)^{-1}g\). By \cref{eq:load-identity}, for every \(w\in X_h\) we have
	\begin{align*}
		a_s(R_hz,R_hw) =(F,R_hw)_H+\kappa s^2(R_hz,R_hw)_H =(g,R_hw)_H.
	\end{align*}
	Thus, \(R_hz\) is exactly the unique Galerkin solution \(U_h\) of \cref{eq:screened-galerkin}. Consequently, for the reconstruction error \(e=U-R_hz\), we have
	\begin{align}\label{eq:residual-cancellation}
		AU = g - \kappa s^2 R_hz = F - \kappa s^2 e.
	\end{align}
	
	\emph{Step 2: compare the Carleman terms.}
	We now apply the continuous Carleman estimate to the shifted continuous solution \(U\). By \cref{eq:residual-cancellation}, we bound the residual term as
	\begin{align*}
		s^{-\gamma}\lVert\theta AU\rVert_{L^2(\Omega)}^2
		\leq Cs^{-\gamma}\lVert\theta F\rVert_{L^2(\Omega)}^2
		+Cs^{4-\gamma}\lVert\theta e\rVert_{L^2(\Omega)}^2.
	\end{align*}
	By \cref{eq:energy-compatibility,eq:observation-compatibility,eq:continuousabstract} and the triangle inequality for \(\mathcal{O}_s\) (noting that \(\mathcal{O}_s(e)\leq C\mathcal{N}_{\theta,\Omega}(e)\)), we obtain
	\begin{align*}
		\mathcal{E}_{s,h}(z)
		\leq C\bigl(s^{-\gamma}\lVert \theta F\rVert_{L^2(\Omega)}^2
		+\mathcal{O}_{s,h}(z)^2
		+\mathcal{N}_{\theta,\Omega}(e)^2
		+s^{4-\gamma}\lVert \theta e\rVert_{L^2(\Omega)}^2\bigr).
	\end{align*}
	To control the error terms \(\mathcal{N}_{\theta,\Omega}(e)\) and \(\lVert\theta e\rVert_{L^2(\Omega)}\), we first bound the auxiliary load. By \cref{eq:aux-load,eq:energy-compatibility} and $\mathcal{E}_s\asymp\mathcal{N}_{\theta,\Omega}^2$,
	\begin{align*}
		\lVert \theta g\rVert_{L^2(\Omega)}^2
		&\leq C\lVert \theta F\rVert_{L^2(\Omega)}^2+Cs^4\lVert\theta R_hz\rVert_{L^2(\Omega)}^2
		\leq C\lVert \theta F\rVert_{L^2(\Omega)}^2+Cs^2\mathcal{E}_{s,h}(z).
	\end{align*}
	Substituting the comparison bounds from \cref{lem:comparison} into the energy inequality, we obtain the preabsorption estimate
	\begin{align}\label{eq:preabsorption}
		\mathcal{E}_{s,h}(z)
		&\leq C\bigl(s^{-\gamma}+\alpha(h,s)^2
		+s^{4-\gamma}\alpha(h,s)^2\alpha^*(h,s)^2\bigr)
		\lVert \theta F\rVert_{L^2(\Omega)}^2\notag\\
		&\quad+C\mathcal{O}_{s,h}(z)^2+C\bigl(s^2\alpha(h,s)^2
		+s^{6-\gamma}\alpha(h,s)^2\alpha^*(h,s)^2\bigr)\mathcal{E}_{s,h}(z).
	\end{align}
	
	\emph{Step 3: absorb the comparison error.}
	Fix $C$ as in \cref{eq:preabsorption} and choose \(c>0\) in \cref{eq:transfer-criterion} small enough so that $Cc\leq1/2$. Then the last term on the right-hand side of \cref{eq:preabsorption} can be absorbed into the left-hand side. 
	For the coefficient of \(\lVert\theta F\rVert_{L^2(\Omega)}^2\), we use the criterion \(s^2\alpha^2 \leq c\) to deduce \(s^{-\gamma}+\alpha^2+s^{4-\gamma}\alpha^2(\alpha^*)^2 \leq s^{-\gamma} + cs^{-2} + cs^{-2} \leq Cs^{-\gamma}\), where we used \(s\geq1\) and \(\gamma\leq1\). This proves \cref{eq:exact-load-CE}.
	
	If, in addition, \cref{ass:loads} holds, we choose \(F=\mathcal{L}_h A_hz\). By \cref{eq:load-lifting,eq:stiffnessinterface},
	\begin{align*}
		(F,R_hw)_H=(A_hz,w)_{X_h}=a(R_hz,R_hw),
	\end{align*}
	so \cref{eq:load-identity} is satisfied. Applying \cref{eq:exact-load-CE,eq:load-stability} then yields \cref{eq:discrete-CE}.
\end{proof}

\section{Verification for \texorpdfstring{$P_1$}{P1} finite elements and Cartesian finite differences}\label{sec:realizations}
In this section, we verify the abstract assumptions of \Cref{thm:transfer} for specific numerical schemes. Our goal is to establish the continuous regularity and discrete approximation properties that allow the Carleman estimate to be transferred to the discrete operator without stabilization. We proceed in three steps: first, we fix the finite-element setting and mass forms; second, we verify the continuous weighted regularity and the discrete approximation and testing assumptions; third, we construct the stable load lift and establish a reciprocal-weight duality. We conclude by showing that standard Cartesian centered finite differences fit directly into this framework through their finite-element realization.

Throughout this section, the approximation bounds will be shown to satisfy \(\alpha(h,s),\alpha^*(h,s)\leq Ch\) when \(hs\leq c_0\) for a sufficiently small fixed \(c_0>0\). The continuous Carleman estimate and energy and observation compatibility in \cref{ass:carleman,ass:observation} are deferred to \cref{sec:weights}.

\subsection{Finite-element setting and weighted norms}\label{sec:mass}

We begin by setting up the discrete framework. Let \(D\subset\mathbb{R}^d\) be a fixed bounded convex polytope, and take \(\Omega=D\) in the abstract framework.
Let \(\mathcal{T}_h(D)\) be a uniformly shape-regular family of conforming simplicial triangulations fitted to \(\partial D\). Write \(h_K=\operatorname{diam}K\) and \(h=\max_{K\in\mathcal{T}_h(D)}h_K\).
No quasi-uniformity is assumed.
For a mesh-fitted subdomain \(G\subset D\), let \(\mathcal{T}_h(G)\) consist of the elements in \(\mathcal{T}_h(D)\) whose union is \(\overline G\). Write \(V_h(G)\) for the continuous piecewise affine space on \(\mathcal{T}_h(G)\), and set \(V_{h,0}(G)=V_h(G)\cap H_0^1(G)\).
All spaces are complexified. An inner product without a subscript denotes the \(L^2\) inner product on the domain under consideration.
On \(D\), take \(X_h=V_{h,0}(D)\) and let \(R_h\) be the natural inclusion.

Let \(\{x_i\}\) be the vertices of \(\mathcal{T}_h(G)\), including boundary vertices, and write \(u_i=u_h(x_i)\) for \(u_h\in V_h(G)\), with the same convention for other discrete functions.
For \(u_h,v_h\in V_h(G)\), consistent mass is defined by
\begin{align}\label{eq:consistent}
m_h(u_h,v_h)=\int_G u_h\overline{v_h}\,dx.
\end{align}
For \(q\in L^\infty(G;\mathbb{C})\) with \(\lVert q\rVert_{L^\infty(G)}\leq M\), the associated potential form is
\begin{align}\label{eq:consistent-reaction}
c_h(q;u_h,v_h)=\int_G q\,u_h\overline{v_h}\,dx.
\end{align}
For lumped mass, we instead use
\begin{align}\label{eq:lumped}
m_h(u_h,v_h)=\sum_i m_i u_i\bar v_i,
\end{align}
where \(m_i=\sum\limits_{K\in\mathcal{T}_h(G),\,x_i\in K}|K|/(d+1)\).
For a nodal potential \(q=(q_i)\) with \(q_i\in\mathbb{C}\) and \(|q_i|\leq M\), the potential form is
\begin{align}\label{eq:lumped-reaction}
c_h(q;u_h,v_h)=\sum_i m_i q_i u_i\bar v_i.
\end{align}
All mass weights and forms on \(G\) are assembled over \(\mathcal{T}_h(G)\); their domain dependence is suppressed from the notation. The inner product on \(X_h\) is \((u_h,v_h)_{X_h}=m_h(u_h,v_h)\) with \(G=D\).
The transfer theorem is applied to the Dirichlet Laplacian \(A=-\Delta_D\), with form domain \(H_0^1(D)\); the potential form is used later in
\begin{align}\label{eq:aq}
a_{q,h}(u_h,v_h)=\int_G \nabla u_h\cdot\nabla\overline{v_h}\,dx+c_h(q;u_h,v_h).
\end{align}

Let \(\rho\in C^2(\overline D;\mathbb{R})\), \(s\geq1\), \(\theta=e^{s\rho}\), and \(\sigma=\theta^{\pm1}\).
From this section on, constants may depend on \(D\), \(M\), a fixed bound for \(\lVert\rho\rVert_{C^2(\overline D)}\), the shape-regularity bound, and, for local estimates, fixed domain separations. They are independent of \(h,s\) and local mesh grading.
For \(z_h\in X_h\), the weighted norm in the consistent-mass case is
\begin{align*}
\lVert z_h\rVert_{X_h,\sigma}^2=\int_D \sigma^2|z_h|^2\,dx.
\end{align*}
In the lumped-mass case, we define
\begin{align*}
\lVert z_h\rVert_{X_h,\sigma}^2=\sum_i m_i\sigma(x_i)^2|z_i|^2.
\end{align*}
For nodal data on a set \(E\subset\overline D\), the restricted lumped-mass norm is
\begin{align}\label{eq:restricted-nodal-norm}
\lVert z_h\rVert_{X_h,\sigma,E}^2
=\sum_{x_i\in E}m_i\sigma(x_i)^2|z_i|^2,
\end{align}
where the weights \(m_i\) are assembled on \(D\). The unweighted norms correspond to \(\sigma=1\).

\subsection{Weighted regularity, approximation, and testing}
We first verify the continuous analytic assumptions. Since the shifted problem \((-\Delta+\kappa s^2)U=g\) is coercive, it inherits the elliptic regularity of the Dirichlet Laplacian. The following lemma shows that the weight derivatives generated by \(\sigma=e^{\pm s\rho}\) can be controlled uniformly by the shift.
\begin{lemma}\label{lem:screen-regularity}
	The Dirichlet Laplacian satisfies \cref{ass:weighted,ass:regularity} uniformly in \(s\geq1\) and \(\sigma=e^{\pm s\rho}\), provided \(\kappa\) is sufficiently large.
\end{lemma}
\begin{proof}
	We verify the two assumptions sequentially.
	
	\emph{Weighted form bounds.}
	Let \(U\in H_0^1(D)\) and set \(z=\sigma U\). Expanding the shifted form gives
	\begin{align}\label{eq:screen-identity}
		\operatorname{Re} a_s(U,\sigma^2U)
		=\lVert \nabla z\rVert_{L^2(D)}^2
		+\int_D\bigl(\kappa s^2-|\nabla\log\sigma|^2\bigr)|z|^2\,dx.
	\end{align}
	Choosing \(\kappa>2\lVert \nabla\rho\rVert_{L^\infty(D)}^2+2\), the product rule and \cref{eq:screen-identity} yield the coercivity \cref{eq:form-coercivity} (with \(\theta\) replaced by \(\sigma\)). Since multiplication by \(\sigma^2\) preserves \(H_0^1(D)\), the boundedness \cref{eq:form-bounds} follows directly from the Cauchy--Schwarz inequality.
	
	\emph{Energy estimate.}
	Next, for \(g\in L^2(D)\), the Lax--Milgram theorem and convex-domain Dirichlet regularity give a unique solution \(U\in H^2(D)\cap H_0^1(D)\) to \((-\Delta+\kappa s^2)U=g\).
	Testing this equation with \(\sigma^2U\) and using \cref{eq:screen-identity}, we obtain the weighted energy bound
	\begin{align}\label{eq:screen-weak}
		s\lVert \nabla z\rVert_{L^2(D)}+s^2\lVert z\rVert_{L^2(D)}
		\leq C\lVert \sigma g\rVert_{L^2(D)}.
	\end{align}
	
	\emph{\(H^2\) regularity.}
	Conjugating the equation by \(\sigma\) yields
	\begin{align*}
		-\Delta z+2\nabla\log\sigma\cdot\nabla z
		+\bigl(\kappa s^2+\Delta\log\sigma-|\nabla\log\sigma|^2\bigr)z
		=\sigma g.
	\end{align*}
	Since \(\log\sigma=\pm s\rho\), the triangle inequality and \cref{eq:screen-weak} imply
	\begin{align*}
		\lVert\Delta z\rVert_{L^2(D)}
		\leq\lVert\sigma g\rVert_{L^2(D)}
		+Cs\lVert\nabla z\rVert_{L^2(D)}
		+C(s^2+s)\lVert z\rVert_{L^2(D)}
		\leq C\lVert\sigma g\rVert_{L^2(D)}.
	\end{align*}
	Applying the convex-domain Dirichlet estimate \cite{Grisvard2011} then gives
	\begin{align*}
		\lVert z\rVert_{H^2(D)}\leq C\lVert\Delta z\rVert_{L^2(D)}
		\leq C\lVert\sigma g\rVert_{L^2(D)}.
	\end{align*}
	
	\emph{Weighted regularity bound.}
	Finally, applying the product rule to \(U=\sigma^{-1}z\), and using \cref{eq:screen-weak} and the \(H^2\) estimate on \(z\), we bound the second-order terms:
	\begin{align*}
		\mathcal R_{\sigma,D}(U)
		\leq C\bigl(\lVert D^2z\rVert_{L^2(D)}
		+s\lVert\nabla z\rVert_{L^2(D)}+s^2\lVert z\rVert_{L^2(D)}\bigr)
		\leq C\lVert\sigma g\rVert_{L^2(D)}.
	\end{align*}
	Since the Dirichlet Laplacian is self-adjoint, this proves \cref{eq:screen-reg,eq:screen-reg-adjoint}.
\end{proof}

Next, we verify the discrete numerical assumptions (N2) and (N3). Local quasi-interpolation provides the weighted approximation bound for functions in the continuous domain, while nodal interpolation provides the test function required for the weighted coercivity estimate. Below, \(0<c_0\leq1\) denotes a fixed mesh threshold, chosen sufficiently small after the constants in estimates valid for \(hs\leq1\) have been determined.
\begin{lemma}\label{lem:interp}
	There exist constants \(c_0,C>0\), depending only on the fixed data, such that for \(hs\leq c_0\), every \(v\in H^2(D)\cap H_0^1(D)\) admits a boundary-preserving approximant \(\Pi_hv\in V_{h,0}(D)\) satisfying
	\begin{align}\label{eq:weighted-SZ}
		\mathcal{N}_{\sigma,D}(v-\Pi_hv)\leq Ch\mathcal{R}_{\sigma,D}(v).
	\end{align}
	Moreover, the nodal interpolant \(I_h\) satisfies
	\begin{align}\label{eq:weighted-product}
		\mathcal{N}_{\sigma^{-1},D}\bigl(I_h(\sigma^2z_h)-\sigma^2z_h\bigr)
		\leq Chs\mathcal{N}_{\sigma,D}(z_h),
		\qquad z_h\in V_{h,0}(D).
	\end{align}
\end{lemma}
\begin{proof}
	We prove the two estimates by standard finite element interpolation techniques.
	
	\emph{Step 1: Weighted approximation.}
	For \(d\geq2\), let \(\Pi_h\) be the boundary-preserving Scott--Zhang quasi-interpolant \cite{SZ}. 
	Set \(e=v-\Pi_hv\). For each element \(K\), let \(\omega_K\) be the union of all elements sharing a vertex with \(K\). Standard local approximation gives
	\begin{align*}
		\lVert e\rVert_{L^2(K)}+h_K\lVert\nabla e\rVert_{L^2(K)}
		\leq Ch_K^2\lVert v\rVert_{H^2(\omega_K)}.
	\end{align*}
	Conformity and shape regularity guarantee bounded patch overlap, and ensure the diameters of elements within each patch are comparable.
	To control the weighted norms, we freeze the weight at a point \(x_K\in K\). The bound \(|\nabla\log\sigma|\leq Cs\) and \(hs\leq1\) imply that the weight is uniformly comparable on \(\omega_K\):
	\begin{align}\label{eq:patch-weight-comparison}
		C^{-1}\leq\frac{\sigma(x)}{\sigma(x_K)}\leq C,
		\qquad x\in\omega_K.
	\end{align}
	Using the local estimate and \(s\geq1\), we obtain
	\begin{align*}
		\lVert\sigma e\rVert_{L^2(K)}+h_K\lVert\sigma\nabla e\rVert_{L^2(K)}
		&\leq C\sigma(x_K)\bigl(\lVert e\rVert_{L^2(K)}+h_K\lVert\nabla e\rVert_{L^2(K)}\bigr)\\
		&\leq Ch_K^2\sigma(x_K)\lVert v\rVert_{H^2(\omega_K)}
		\leq Ch_K^2\mathcal R_{\sigma,\omega_K}(v).
	\end{align*}
	Summing over all elements, and using the bounded patch overlap, \(h_K\leq h\), and \(hs\leq1\), we obtain
	\begin{align*}
		\mathcal N_{\sigma,D}(e)^2
		&=\sum_{K\in\mathcal T_h(D)}
		\bigl(\lVert\sigma\nabla e\rVert_{L^2(K)}^2+s^2\lVert\sigma e\rVert_{L^2(K)}^2\bigr)\\
		&\leq C\sum_{K\in\mathcal T_h(D)}h_K^2(1+s^2h_K^2)\mathcal R_{\sigma,\omega_K}(v)^2
		\leq Ch^2\mathcal R_{\sigma,D}(v)^2.
	\end{align*}
	Taking square roots proves \cref{eq:weighted-SZ}. In one dimension, the same calculation applies with \(K\) in place of \(\omega_K\).
	
	\emph{Step 2: Local interpolation error.}
	We now establish the second estimate, which concerns the interpolation of the product \(\sigma^2z_h\). 
	Set \(F=\sigma^2z_h\) and fix an element \(K\). Since \(z_h\) is affine on \(K\), we have \(D^2z_h=0\), and thus
	\begin{align}\label{eq:weighted-product-hessian}
		D^2F=z_hD^2(\sigma^2)
		+\nabla(\sigma^2)\otimes\nabla z_h
		+\nabla z_h\otimes\nabla(\sigma^2).
	\end{align}
	Since \(\sigma^2=e^{\pm2s\rho}\) and \(s\geq1\), the derivatives of the weight satisfy
	\begin{align*}
		|\nabla(\sigma^2)|+s^{-1}|D^2(\sigma^2)|\leq Cs\sigma^2.
	\end{align*}
	To bound the interpolation error, let \(p_K(x)=F(x_K)+\nabla F(x_K)\cdot(x-x_K)\) be the first-order Taylor polynomial at \(x_K\), and set \(r_K=F-p_K\). Taylor's theorem gives
	\begin{align}\label{eq:product-taylor-remainder}
		\lVert r_K\rVert_{L^\infty(K)}+h_K\lVert\nabla r_K\rVert_{L^\infty(K)}
		\leq Ch_K^2\lVert D^2F\rVert_{L^\infty(K)}.
	\end{align}
	Let \(x_i\) be the vertices of \(K\) and \(\lambda_i\) their barycentric coordinates. Since \(I_hp_K=p_K\), the interpolation error is
	\begin{align*}
		(I_hF-F)|_K=\sum_{x_i\in K}r_K(x_i)\lambda_i-r_K.
	\end{align*}
	Using \(|\nabla\lambda_i|\leq Ch_K^{-1}\) together with \cref{eq:product-taylor-remainder}, we infer
	\begin{align*}
		\lVert I_hF-F\rVert_{L^\infty(K)}
		+h_K\lVert\nabla(I_hF-F)\rVert_{L^\infty(K)}
		\leq Ch_K^2\lVert D^2F\rVert_{L^\infty(K)}.
	\end{align*}
	Next, we convert this \(L^\infty\) bound into a weighted \(L^2\) bound. By \cref{eq:patch-weight-comparison,eq:weighted-product-hessian} and norm equivalence for affine functions on \(K\),
	\begin{align*}
		|K|^{1/2}\lVert\sigma^{-1}\rVert_{L^\infty(K)}\lVert D^2F\rVert_{L^\infty(K)}
		&\leq C|K|^{1/2}\sigma(x_K)
		\bigl(s\lVert\nabla z_h\rVert_{L^\infty(K)}+s^2\lVert z_h\rVert_{L^\infty(K)}\bigr)\\
		&\leq C\bigl(s\lVert\sigma\nabla z_h\rVert_{L^2(K)}
		+s^2\lVert\sigma z_h\rVert_{L^2(K)}\bigr).
	\end{align*}
	Setting \(\eta_h=I_h(\sigma^2z_h)-\sigma^2z_h\) and combining the above bounds yields
	\begin{align}\label{eq:product-local}
		\lVert \sigma^{-1}\eta_h\rVert_{L^2(K)}
		+h_K\lVert \sigma^{-1}\nabla\eta_h\rVert_{L^2(K)}
		\leq Ch_K^2\bigl(s\lVert \sigma\nabla z_h\rVert_{L^2(K)}
		+s^2\lVert \sigma z_h\rVert_{L^2(K)}\bigr).
	\end{align}
	
	\emph{Step 3: Global weighted testing estimate.}
	Finally, we sum the local estimate \cref{eq:product-local} over all elements. Using \(sh_K\leq1\), we have
	\begin{align*}
		\mathcal N_{\sigma^{-1},K}(\eta_h)^2
		\leq Ch_K^2s^2(1+s^2h_K^2)\mathcal N_{\sigma,K}(z_h)^2
		\leq Ch_K^2s^2\mathcal N_{\sigma,K}(z_h)^2.
	\end{align*}
	Summing over the mesh gives
	\begin{align*}
		\mathcal N_{\sigma^{-1},D}(\eta_h)^2
		=\sum_{K\in\mathcal T_h(D)}\mathcal N_{\sigma^{-1},K}(\eta_h)^2
		\leq Cs^2\sum_{K\in\mathcal T_h(D)}h_K^2\mathcal N_{\sigma,K}(z_h)^2
		\leq Ch^2s^2\mathcal N_{\sigma,D}(z_h)^2.
	\end{align*}
	Taking square roots proves \cref{eq:weighted-product}.
\end{proof}

Thus \cref{Papprox} holds with \(\alpha(h,s)=\alpha^*(h,s)=Ch\). Choosing \(w_h=I_h(\theta^2v_h)\) verifies \cref{Ptest} after decreasing \(c_0\) so that the factor \(Chs\) in \cref{eq:weighted-product} is at most the ratio in \cref{eq:testsuper}. The same argument applies with \(\theta\) replaced by \(\theta^{-1}\).

\subsection{Load lifting and weighted mass equivalence}

We now verify the stable load lifting assumption (N1) and establish a duality estimate that will be crucial for local weighted solvability and CGO construction later. We first note that local polynomial norm equivalence and local weight comparability give
\begin{align}\label{eq:weighted-mass-equivalence}
\lVert z_h\rVert_{X_h,\sigma}\asymp\lVert \sigma z_h\rVert_{L^2(D)},
\qquad z_h\in X_h,\quad hs\leq c_0.
\end{align}
For consistent mass, take \(\mathcal{L}_hf_h=f_h\) for \(f_h\in X_h\).
For lumped mass, an elementwise biorthogonal basis represents the nodal pairing. Define the broken affine lift by
\begin{align}\label{eq:broken-lift}
(\mathcal{L}_hf_h)|_K
=\sum_{x_i\in K}f_i\bigl((d+2)\lambda_i-1\bigr),
\end{align}
where the sum is over the vertices of \(K\), \(\lambda_i\) are its barycentric coordinates, and \(f_i=0\) at boundary vertices of \(D\).
The identity
\begin{align*}
\int_K\bigl((d+2)\lambda_i-1\bigr)\lambda_j\,dx
=\frac{|K|}{d+1}\delta_{ij}
\end{align*}
proves \cref{eq:load-lifting} after summation over the mesh.
Polynomial norm equivalence and weight comparability on each element give \cref{eq:load-stability}, completing the verification of \cref{ass:loads}.

The following reciprocal-weight duality estimate will be used to control the potential term and to construct a weighted right inverse in \cref{sec:cgo}.
\begin{lemma}\label{lem:mass-duality}
	For consistent or lumped mass and \(hs\leq c_0\) (with \(c_0>0\) sufficiently small), we have
	\begin{align}\label{eq:mass-dual}
		\lVert z_h\rVert_{X_h,\sigma}\asymp
		\sup_{0\ne v_h\in X_h}
		\frac{|(z_h,v_h)_{X_h}|}{\lVert v_h\rVert_{X_h,\sigma^{-1}}},\qquad z_h\in X_h.
	\end{align}
	The potential operator \(C_h(q):X_h\to X_h\), defined by \((C_h(q)z_h,v_h)_{X_h}=c_h(q;z_h,v_h)\), satisfies
	\begin{align}\label{eq:potential-weighted}
		\lVert C_h(q)z_h\rVert_{X_h,\sigma}\leq CM\lVert z_h\rVert_{X_h,\sigma}.
	\end{align}
\end{lemma}
\begin{proof}
The weighted Cauchy--Schwarz inequality bounds the supremum in \cref{eq:mass-dual} by \(\lVert z_h\rVert_{X_h,\sigma}\).
For lumped mass, equality is attained by \(v_i=\sigma(x_i)^2z_i\).
For consistent mass, choose \(v_h=I_h(\sigma^2z_h)\).
Using \cref{eq:product-local} and the local inverse inequality
\(h_K\lVert \nabla z_h\rVert_{L^2(K)}\leq C\lVert z_h\rVert_{L^2(K)}\), we obtain
\begin{align*}
	\lVert \sigma^{-1}(v_h-\sigma^2z_h)\rVert_{L^2(D)}
	\leq C\bigl(hs+(hs)^2\bigr)\lVert \sigma z_h\rVert_{L^2(D)}.
\end{align*}
Consequently,
\begin{align*}
	\operatorname{Re}(z_h,v_h)_{X_h}\geq(1-Chs)\lVert \sigma z_h\rVert_{L^2(D)}^2.
\end{align*}
The same estimate and the triangle inequality give
\begin{align*}
	\lVert \sigma^{-1}v_h\rVert_{L^2(D)}\leq C\lVert \sigma z_h\rVert_{L^2(D)}.
\end{align*}
Choose \(c_0\) so that \(Chs\leq1/2\). These bounds show that \(\lVert z_h\rVert_{X_h,\sigma}\) is bounded by a constant times the supremum in \cref{eq:mass-dual}, proving the reverse inequality.
Finally, \cref{eq:consistent-reaction,eq:lumped-reaction} give
\(|c_h(q;z_h,v_h)|\leq M\lVert z_h\rVert_{X_h,\sigma}\lVert v_h\rVert_{X_h,\sigma^{-1}}\);
apply \cref{eq:mass-dual} to \(C_h(q)z_h\).
\end{proof}

\subsection{Cartesian finite differences}\label{sec:cartesian}
Finally, we observe that standard Cartesian centered finite differences fit directly into this finite-element framework. Let \(D\) be a box with mutually comparable uniform coordinate spacings \(h_j\).
Index its grid vertices by multi-indices \(i\in\mathbb{Z}^d\), and let \(e_j\) denote the \(j\)th coordinate vector, so that \(x_{i\pm e_j}=x_i\pm h_je_j\). Write \(z_i=z_h(x_i)\).
Triangulate each grid box by the Kuhn simplices obtained from the coordinate permutations.
For lumped mass, an interior node has \(m_i=\prod\limits_{j=1}^dh_j\), and the resulting \(P_1\) operator is the centered finite-difference Laplacian:
\begin{align}\label{eq:cartesian-laplacian}
(A_hz_h)_i=\sum_{j=1}^d\frac{2z_i-z_{i+e_j}-z_{i-e_j}}{h_j^2}.
\end{align}
Indeed, on each simplex the coordinate derivatives of an affine function are its successive coordinate-edge differences divided by the corresponding spacings.
Summing the simplex energies gives coefficient \((\prod\limits_{k=1}^dh_k)/h_j^2\) on each interior coordinate edge in direction \(e_j\) and zero on noncoordinate edges.
Dividing these stiffness coefficients by \(m_i\) gives \cref{eq:cartesian-laplacian}; \cref{eq:lumped-reaction} gives \((C_h(q)z_h)(x_i)=q_iz_i\).
The maximum simplex diameter is comparable to \(\max_jh_j\), so \cref{eq:weighted-SZ,eq:weighted-product,eq:load-stability} apply with \(h\asymp\max_jh_j\).

The preceding estimates verify the load, approximation, and testing assumptions for both finite-element mass forms and for Cartesian differences.
\begin{proposition}\label{prop:transfer-realizations}
	For the finite-element spaces and weights in \cref{sec:mass}, there exist constants \(c_0,C>0\), depending only on the fixed data, such that the assumptions \cref{ass:loads,Papprox,Ptest} hold for both \(\theta\) and \(\theta^{-1}\) with \(\alpha(h,s),\alpha^*(h,s)\leq Ch\), whenever \(hs\leq c_0\) and consistent or lumped mass is used. The same holds for the Cartesian operator in \cref{eq:cartesian-laplacian}.
\end{proposition}
\begin{proof}
	We verify the three numerical assumptions in order.
	
	First, \cref{ass:loads} holds by the constructions in the preceding subsection (specifically the identity for consistent mass and the broken affine lift for lumped mass), while \cref{eq:stiffnessinterface} follows directly from the definition of \(A_h\). 
	
	Next, \cref{Papprox} is a direct consequence of \cref{lem:interp}, which provides the bounds with \(\alpha(h,s),\alpha^*(h,s)\leq Ch\). 
	
	Finally, to verify the weighted testing assumption \cref{Ptest}, we recall the bound \cref{eq:weighted-product} from \cref{lem:interp}. By choosing the mesh threshold \(c_0>0\) sufficiently small, the factor \(Chs\) is controlled so that \cref{eq:testsuper} is satisfied.
	
	The same arguments apply to the Cartesian operator via its lumped-mass finite-element realization in \cref{eq:cartesian-laplacian}.
\end{proof}

\section{Transferred Carleman estimates}\label{sec:weights}
Having verified the numerical assumptions in \cref{prop:transfer-realizations}, we now apply \cref{thm:transfer,cor:budgets} to two specific continuous Carleman estimates. Each choice of the weight $\theta$ serves a distinct purpose in the subsequent inverse problem analysis:
\begin{itemize}
\item \emph{Linear weights} use $\gamma=0$ and yield the parameter range \(s\leq ch^{-2/3}\). They provide a coercive estimate that serves as the foundation for constructing exact discrete CGO solutions in \cref{sec:cgo}.
\item \emph{Radial weights} use $\gamma=1$ and yield the slightly wider range \(s\leq ch^{-4/5}\). They yield quantitative continuation estimates (three-ball inequalities) that we use to propagate smallness across the domain in \cref{sec:local}.
\end{itemize}
Throughout this section, let $q$ be a bounded complex potential as in \cref{eq:consistent-reaction} for consistent mass, or a bounded nodal potential as in \cref{eq:lumped-reaction} for lumped mass. We define the discrete Schr\"odinger operator $A_{q,h}=A_h+C_h(q)$ on $X_h=V_{h,0}(D)$, where \(A_h\) is the discrete Laplacian and \(C_h(q)\) is the potential operator. Its adjoint with respect to $(\cdot,\cdot)_{X_h}$ is $A_{q,h}^*=A_{\bar q,h}$.

\subsection{Linear weights and local coercivity}
We begin with the linear weight $\theta=e^{s\mathbf{a}\cdot x}$, where $|\mathbf{a}|=1$. Our goal is to control the weighted energy of a state by its equation residual, assuming the observation is made in an outer collar disjoint from the region of interest. This will later provide the solvability estimate for the CGO construction. The continuous estimate is standard and relies on the commutation of the conjugated Laplacian with its symmetric and skew-symmetric parts.
\begin{lemma}\label{lem:linear-cont}
There are an outer collar $O\subset D$, with $\overline O\cap\overline U=\varnothing$, and a constant $C>0$ such that
\begin{align}\label{eq:linear-collar}
	\mathcal{N}_{\theta,D}(v)^2
	\leq C\bigl(\lVert \theta\Delta v\rVert_{L^2(D)}^2
	+\mathcal{N}_{\theta,O}(v)^2\bigr),
	\qquad v\in H^2(D)\cap H_0^1(D),
\end{align}
uniformly for $s\geq1$ and $|\mathbf{a}|=1$.
\end{lemma}
\begin{proof}
	We proceed in two steps: first establishing the estimate for compactly supported smooth functions, and then extending it to \(H^2(D)\cap H_0^1(D)\) via a cutoff argument.
	
	\emph{Step 1: Compact support estimate.}
	For $v\in C_c^\infty(D)$, extend $z=\theta v$ by zero to \(\mathbb{R}^d\). Because the symmetric and skew-symmetric parts of \(P_s=\theta(-\Delta)\theta^{-1}=-\Delta-s^2+2s\mathbf{a}\cdot\nabla\) commute, we have
	\begin{align*}
		\lVert P_sz\rVert_{L^2(\mathbb{R}^d)}^2
		=\lVert (-\Delta-s^2)z\rVert_{L^2(\mathbb{R}^d)}^2
		+4s^2\lVert \mathbf{a}\cdot\nabla z\rVert_{L^2(\mathbb{R}^d)}^2.
	\end{align*}
	Poincar\'e's inequality along lines parallel to $\mathbf{a}$ yields $s\lVert z\rVert_{L^2(\mathbb{R}^d)}\leq C\lVert P_sz\rVert_{L^2(\mathbb{R}^d)}$, uniformly in $\mathbf{a}$, since each intersection with \(D\) has length at most \(\operatorname{diam}D\).
	Combining this with the identity $\lVert\nabla z\rVert_{L^2(\mathbb{R}^d)}^2=\operatorname{Re}(P_sz,z)_{L^2(\mathbb{R}^d)}+s^2\lVert z\rVert_{L^2(\mathbb{R}^d)}^2$ and the product rule, we obtain
	\begin{align}\label{eq:linear-compact}
		\mathcal{N}_{\theta,D}(v)^2\leq C\lVert\theta\Delta v\rVert_{L^2(D)}^2,
		\qquad v\in C_c^\infty(D).
	\end{align}
	
	\emph{Step 2: Extension via cutoff.}
	Choose $\chi\in C_c^\infty(D)$ equal to one near $\overline U$, and an outer collar $O$ containing $D\cap\operatorname{supp}(1-\chi)$ with $\overline O\cap\overline U=\varnothing$.
	By density, \cref{eq:linear-compact} applies to $\chi v$ for any $v\in H^2(D)\cap H_0^1(D)$. 
	The remaining error terms generated by the cutoff can be bounded in terms of the collar. Specifically, the commutator $[\Delta,\chi]v=2\nabla\chi\cdot\nabla v+(\Delta\chi)v$ and the leftover term \((1-\chi)v\) are controlled by $C\mathcal{N}_{\theta,O}(v)$ in the weighted energy and \(L^2\) norms, respectively. Combining these with the estimate for \(\chi v\) proves \cref{eq:linear-collar}.
\end{proof}
We now transfer this estimate to discrete states supported in \(U\). Let
\begin{align*}
Z_h(U)=\{z_h\in V_{h,0}(D):\operatorname{supp} z_h\subset\overline U\}.
\end{align*}
The set $U$ need not be mesh fitted; $Z_h(U)$ is simply a subspace of the finite-element space on $D$. By applying the transfer theorem and absorbing the lower-order potential terms, we obtain a uniform coercivity estimate for the Schr\"odinger operator.
\begin{theorem}\label{thm:linear-discrete}
	For either mass realization on the fitted, uniformly shape-regular meshes of \cref{sec:realizations}, there exist constants \(s_0,c,C,h_0>0\), depending only on the fixed data and the separation of \(U\) from \(\partial D\), such that, uniformly in \(|\mathbf{a}|=1\), the estimate
	\begin{align}\label{eq:linear-discrete}
		s\lVert z_h\rVert_{X_h,\theta}+\lVert \theta\nabla z_h\rVert_{L^2(D)}
		\leq C\lVert A_{q,h}z_h\rVert_{X_h,\theta}
	\end{align}
	holds for all \(0<h\leq h_0\), \(s_0\leq s\leq ch^{-2/3}\), and \(z_h\in Z_h(U)\).
\end{theorem}
\begin{proof}
	We apply the abstract transfer theorem \cref{thm:transfer} and the parameter budget \cref{cor:budgets} to the continuous estimate \cref{eq:linear-collar}. 
	
	\emph{Step 1: Set up the transfer.} 
	Set \(\gamma=0\) and choose the continuous energy and observation terms as \(\mathcal{E}_s=\mathcal{N}_{\theta,D}^2\) and \(\mathcal{O}_s=\mathcal{N}_{\theta,O}\), respectively. The corresponding discrete terms are \(\mathcal{E}_{s,h}(z_h)=\mathcal{N}_{\theta,D}(z_h)^2\) and \(\mathcal{O}_{s,h}(z_h)=\mathcal{N}_{\theta,O}(z_h)\), which satisfy the compatibility assumption \cref{ass:observation}. The numerical assumptions are provided by \cref{prop:transfer-realizations}.
	
	\emph{Step 2: Vanishing observation.} 
	For any state \(z_h\in Z_h(U)\), the support condition and \(\overline O\cap\overline U=\varnothing\) imply that the observation term vanishes. Consequently, \cref{eq:discrete-CE} reduces to
	\begin{align*}
		\mathcal{N}_{\theta,D}(z_h)\leq C\lVert A_hz_h\rVert_{X_h,\theta}.
	\end{align*}
	
	\emph{Step 3: Absorb the potential.} 
	To relate this to the Schr\"odinger operator \(A_{q,h}\), we use the decomposition \(A_hz_h=A_{q,h}z_h-C_h(q)z_h\). By \cref{eq:potential-weighted,eq:weighted-mass-equivalence}, the potential term satisfies
	\begin{align*}
		\lVert C_h(q)z_h\rVert_{X_h,\theta}
		\leq CM\lVert z_h\rVert_{X_h,\theta}
		\leq CM s^{-1}\mathcal{N}_{\theta,D}(z_h).
	\end{align*}
	Choosing \(s_0\) sufficiently large, this lower-order contribution can be absorbed into the left-hand side.
	
	\emph{Step 4: Final bound.} 
	A final use of \cref{eq:weighted-mass-equivalence} recovers the gradient and \(L^2\) terms from \(\mathcal{N}_{\theta,D}(z_h)\), yielding \cref{eq:linear-discrete}. Finally, we decrease \(h_0\) if necessary so that the condition \(s_0\leq ch_0^{-2/3}\) holds.
\end{proof}

The residual norm in \cref{eq:linear-discrete} is taken on the whole auxiliary domain $D$.
In the nodal basis, $A_h$ has matrix $\mathsf{M}_h^{-1}\mathsf{K}_h$, where $\mathsf{M}_h$ and $\mathsf{K}_h$ are the mass and stiffness matrices. For consistent mass, $\mathsf{M}_h^{-1}$ is generally dense, so $A_{q,h}$ need not preserve supports. We therefore impose local equations variationally, as in \cref{eq:local-equation}.

\subsection{Radial weights and local observation}
For the quantitative continuation estimates in \cref{sec:local}, we need a Carleman estimate that observes a state only on a small interior ball. For this purpose, we switch to the radial weight $\theta=e^{s\rho}$ with $\rho(x)=e^{-\lambda|x-x_*|^2}$, which is tailored to the geometry of concentric balls. We also allow a bounded first-order drift term, which arises naturally in the localization process. Let $b\in L^\infty(D;\mathbb{C}^d)$ satisfy $\lVert b\rVert_{L^\infty(D)}\leq M_1$, and define, for $u_h,v_h\in V_h(D)$,
\begin{align}\label{eq:drift-form}
a_{b,q,h}(u_h,v_h)=a_{q,h}(u_h,v_h)+\int_D(b\cdot\nabla u_h)\overline{v_h}\,dx.
\end{align}
The integral in \cref{eq:drift-form} is used for both mass forms. On $X_h=V_{h,0}(D)$, we define the operator \(A_{b,q,h}\) by
\begin{align*}
(A_{b,q,h}z_h,v_h)_{X_h}=a_{b,q,h}(z_h,v_h),\qquad z_h,v_h\in X_h.
\end{align*}
In the radial, local energy, and continuation estimates, the fixed data also include $M_1$; the drift may depend on $h$ subject to this uniform bound.

Fix $x_*\in D$ and a ball $B_*\Subset D$ centered at $x_*$. For a sufficiently large fixed $\lambda$, set
\begin{align}\label{eq:radial-weight}
\rho(x)=e^{-\lambda|x-x_*|^2}.
\end{align}
The following continuous estimate is proved in \cref{app:radial}. It shows that the weighted energy in the whole domain can be controlled by the Laplacian residual and the weighted \(L^2\) norm observed on the inner ball $B_*$.
\begin{align}\label{eq:radial-cont}
\mathcal{N}_{\theta,D}(v)^2
\leq C\bigl(s^{-1}\lVert \theta\Delta v\rVert_{L^2(D)}^2
+s^2\lVert \theta v\rVert_{L^2(B_*)}^2\bigr),
\qquad v\in H^2(D)\cap H_0^1(D).
\end{align}

The following theorem transfers \cref{eq:radial-cont} to the discrete operator \(A_{b,q,h}\). Its exact-load formulation is essential for the localized equations in \cref{sec:local}, where the strong discrete residual may not be well-defined.
\begin{theorem}\label{thm:radial-discrete}
	For either mass realization with the potential forms from \cref{eq:consistent-reaction,eq:lumped-reaction}, and for \(b\) and \(\theta\) as above, there exist constants \(s_0,c,C,h_0>0\), depending only on the fixed data, such that for all \(0<h\leq h_0\) and \(s_0\leq s\leq ch^{-4/5}\), the estimate
	\begin{align}\label{eq:radial-discrete}
		\lVert\theta\nabla z_h\rVert_{L^2(D)}^2+s^2\lVert\theta z_h\rVert_{L^2(D)}^2
		\leq C\bigl(s^{-1}\lVert A_{b,q,h}z_h\rVert_{X_h,\theta}^2
		+s^2\lVert\theta z_h\rVert_{L^2(B_*)}^2\bigr)
	\end{align}
	holds for all \(z_h\in V_{h,0}(D)\). Moreover, if \(F\in L^2(D)\) satisfies \(a_{b,q,h}(z_h,v_h)=(F,v_h)_{L^2(D)}\) for all \(v_h\in V_{h,0}(D)\), then
	\begin{align}\label{eq:radial-load}
		\mathcal{N}_{\theta,D}(z_h)^2
		\leq C\bigl(s^{-1}\lVert\theta F\rVert_{L^2(D)}^2
		+s^2\lVert\theta z_h\rVert_{L^2(B_*)}^2\bigr).
	\end{align}
\end{theorem}
\begin{proof}
	We prove the two estimates sequentially.
	
	\emph{Step 1: The first estimate.} 
	We apply \cref{thm:transfer,cor:budgets} to the unperturbed Dirichlet Laplacian and the continuous estimate \cref{eq:radial-cont}, with \(\gamma=1\). The energy and observation terms are chosen as \(\mathcal{E}_{s,h}(z_h)=\mathcal{N}_{\theta,D}(z_h)^2\) and \(\mathcal{O}_{s,h}(z_h)=s\lVert\theta z_h\rVert_{L^2(B_*)}\), which satisfy \cref{ass:observation}. The numerical assumptions are supplied by \cref{prop:transfer-realizations}.
	
	\emph{Step 2: Absorbing the lower-order terms.} 
	To relate \(A_h\) to the perturbed operator \(A_{b,q,h}\), we use the decomposition
	\begin{align*}
		A_hz_h = A_{b,q,h}z_h - (A_{b,q,h}-A_{q,h})z_h - C_h(q)z_h.
	\end{align*}
	By \cref{lem:mass-duality,eq:weighted-mass-equivalence}, the weighted Cauchy--Schwarz inequality yields
	\begin{align*}
		\lVert(A_{b,q,h}-A_{q,h})z_h\rVert_{X_h,\theta} \leq CM_1\lVert\theta\nabla z_h\rVert_{L^2(D)}.
	\end{align*}
	Substituting this and the potential bound \cref{eq:potential-weighted} into \cref{eq:discrete-CE}, the drift and potential contributions are bounded by
	\begin{align}\label{eq:lower-order-absorption}
		C\biggl(\frac{M_1^2}{s}+\frac{M^2}{s^3}\biggr)\mathcal N_{\theta,D}(z_h)^2.
	\end{align}
	Choosing \(s_0\) sufficiently large allows us to absorb this term into the left-hand side of \cref{eq:discrete-CE}, proving the first estimate \cref{eq:radial-discrete}.
	
	\emph{Step 3: The load formulation.} 
	We now prove \cref{eq:radial-load}. Define a load \(F_{q,h}\) representing the potential term: let \(F_{q,h}=qz_h\) for consistent mass, and let \(F_{q,h}\) be the broken lift \cref{eq:broken-lift} of \((q_iz_i)_i\) for lumped mass. In both cases, we have
	\begin{align*}
		(F_{q,h},v_h)_{L^2(D)}=c_h(q;z_h,v_h),\qquad v_h\in V_{h,0}(D).
	\end{align*}
	By the potential bound and \cref{eq:load-stability,eq:weighted-mass-equivalence}, we deduce
	\begin{align*}
		\lVert\theta F_{q,h}\rVert_{L^2(D)}\leq CM\lVert\theta z_h\rVert_{L^2(D)}.
	\end{align*}
	The assumed variational equation can thus be rewritten as
	\begin{align*}
		(\nabla z_h,\nabla v_h)_{L^2(D)}
		=(F-b\cdot\nabla z_h-F_{q,h},v_h)_{L^2(D)},\qquad v_h\in V_{h,0}(D).
	\end{align*}
	Applying \cref{eq:exact-load-CE} to this equation and absorbing the lower-order terms as in \cref{eq:lower-order-absorption}, we obtain \cref{eq:radial-load}. The proof concludes by choosing \(h_0\) small enough to ensure \(s_0\leq ch_0^{-4/5}\).
\end{proof}

\subsection{Local equations and a Caccioppoli estimate}

The CGO and continuation arguments both rely on controlling the local energy of a discrete function that satisfies an equation only on a subdomain. To formalize this, for an open set $G\subset D$, we define the local test space
\begin{align*}
\mathcal V_h(G)=\{v_h\in V_{h,0}(D):\operatorname{supp}v_h\Subset G\}.
\end{align*}
For $u_h\in V_h(D)$ and $f\in L^2(D)$, we say that the discrete equation holds in $G$ if
\begin{align}\label{eq:local-equation}
a_{b,q,h}(u_h,v_h)=(f,v_h)_{L^2(D)},\qquad
v_h\in\mathcal V_h(G).
\end{align}
Since the test functions are compactly supported in $G$, the evaluation of the forms only involves the region $G$, allowing us to use cutoffs separated from the boundary of the region where the equation is imposed. The following Caccioppoli estimate is a standard tool in local finite-element theory; see \cite[Section~2]{NS1974}. We provide the proof to track the dependence on the measurable drift and the two potential realizations.
\begin{lemma}\label{lem:caccioppoli}
	Let $G_0\Subset G_1\Subset G$ be fixed open sets. There are \(C,h_0>0\), depending only on the fixed data and these sets, such that every solution of \cref{eq:local-equation} satisfies, for \(0<h\leq h_0\),
	\begin{align}\label{eq:caccioppoli}
		\lVert u_h\rVert_{H^1(G_0)}
		\leq C\bigl(\lVert u_h\rVert_{L^2(G_1)}+\lVert f\rVert_{L^2(G_1)}\bigr).
	\end{align}
\end{lemma}
\begin{proof}
	We test the local equation with a smoothed, localized version of \(u_h\).
	
	\emph{Step 1: Test function and interpolation error.}
	Choose a smooth cutoff \(0\leq\eta\leq1\) equal to one near \(G_0\) and supported strictly inside \(G_1\). Take \(h_0\leq1\) small enough that \(I_h(\eta^2u_h)\) is supported in \(G_1\) for all \(h\leq h_0\); then \(I_h(\eta^2u_h)\in\mathcal V_h(G)\).
	Set \(e_h=I_h(\eta^2u_h)-\eta^2u_h\) and \(\eta_K=\sup_K|\eta|\). Elementwise superapproximation gives
	\begin{align*}
		\lVert e_h\rVert_{L^2(K)}+h_K\lVert\nabla e_h\rVert_{L^2(K)}
		\leq Ch_K^2\bigl(\eta_K\lVert\nabla u_h\rVert_{L^2(K)}+\lVert u_h\rVert_{L^2(K)}\bigr).
	\end{align*}
	Denote \(X=\lVert\eta\nabla u_h\rVert_{L^2(D)}\) and \(U_*=\lVert u_h\rVert_{L^2(G_1)}\).
	
	\emph{Step 2: Bounds on the individual terms.}
	We estimate the contributions to the form \(a_{b,q,h}(u_h,I_h(\eta^2u_h))\). 
	First, the stiffness term splits as \((\nabla u_h,\nabla I_h(\eta^2u_h))=X^2+(\nabla u_h,\nabla e_h)\). Using \(|\eta_K-\eta|\leq Ch_K\) and the elementwise inverse estimate together with the bound on \(e_h\), we obtain
	\begin{align*}
		|(\nabla u_h,\nabla e_h)_{L^2(D)}|\leq CU_*(X+U_*).
	\end{align*}
	Next, the drift term satisfies
	\begin{align*}
		\Bigl|\int_D(b\cdot\nabla u_h)\overline{I_h(\eta^2u_h)}\,dx\Bigr|
		\leq CM_1\bigl(XU_*+hU_*(X+U_*)\bigr).
	\end{align*}
	Finally, for the potential term, we have \(|c_h(q;u_h,I_h(\eta^2u_h))|\leq CMU_*^2\), which follows from local \(L^2\) stability of nodal multiplication for consistent mass and nodal norm equivalence for lumped mass.
	
	\emph{Step 3: Absorption via Young's inequality.}
	Testing \cref{eq:local-equation} with \(I_h(\eta^2u_h)\) and taking real parts, the principal term is \(X^2\). The cutoff cross term is bounded by \(CU_*X\), and the forcing contributes at most \(C\lVert f\rVert_{L^2(G_1)}U_*\). For \(h\leq1\), applying Young's inequality absorbs the linear terms in \(X\) into \(X^2\), yielding
	\begin{align*}
		X^2 \leq C\bigl(U_*^2+\lVert f\rVert_{L^2(G_1)}^2\bigr).
	\end{align*}
	Since \(\eta=1\) near \(G_0\), we have \(\lVert\nabla u_h\rVert_{L^2(G_0)}\leq X\), and \(\lVert u_h\rVert_{L^2(G_0)}\leq U_*\). Taking square roots proves \cref{eq:caccioppoli}.
\end{proof}

\section{Discrete CGO solutions}\label{sec:cgo}
We now construct discrete complex geometrical optics (CGO) solutions, which will later be used to probe Fourier modes of the potential. Throughout this section, we set $b=0$ and work with the Schr\"odinger form $a_{q,h}$ on $X_h=V_{h,0}(D)$. The construction proceeds in two stages. First, building on the linear Carleman estimate of \cref{thm:linear-discrete}, we derive a local weighted solvability estimate for the discrete Schr\"odinger operator. Second, we take a shifted Galerkin approximation of a cut-off harmonic exponential and use the solvability estimate to correct its local discrete defect exactly. This yields a discrete solution that satisfies the homogeneous equation in a local region and approximates a harmonic exponential there.

\subsection{A local weighted solvability estimate}
The linear Carleman estimate \cref{eq:linear-discrete} provides a coercivity bound for the Schr\"odinger operator when tested against functions supported in $U$. By duality, this yields the following solvability estimate for local equations, which will serve as the key tool for correcting the discrete defect in the CGO construction.
\begin{lemma}\label{lem:right-inverse}
	Let $\theta=e^{s\mathbf{a}\cdot x}$ with $|\mathbf{a}|=1$, and suppose $h,s$ lie in the range of \cref{thm:linear-discrete}. If an antilinear functional $F$ on $Z_h(U)$ satisfies, for some $L\geq0$,
	\begin{align*}
		|F(v_h)|\leq L\lVert v_h\rVert_{X_h,\theta^{-1}},\qquad v_h\in Z_h(U),
	\end{align*}
	then there exists \(r_h\in V_{h,0}(D)\) such that
	\begin{align*}
		a_{q,h}(r_h,v_h)=F(v_h),\qquad v_h\in Z_h(U),
	\end{align*}
	with the weighted bound
	\begin{align}\label{eq:right-inverse}
		\lVert r_h\rVert_{X_h,\theta}\leq Cs^{-1}L.
	\end{align}
\end{lemma}
\begin{proof}
Apply \cref{eq:linear-discrete} with potential $\bar q$ and weight $\theta^{-1}$. By \cref{eq:mass-dual} and the adjoint identity $a_{q,h}(r_h,v_h)=m_h(r_h,A_{\bar q,h}v_h)$, we obtain the discrete inf-sup estimate
\begin{align*}
	s\lVert v_h\rVert_{X_h,\theta^{-1}}
	&\leq C\lVert A_{\bar q,h}v_h\rVert_{X_h,\theta^{-1}}
	\leq C\sup_{0\ne r_h\in X_h}
	\frac{|a_{q,h}(r_h,v_h)|}{\lVert r_h\rVert_{X_h,\theta}},
	\qquad v_h\in Z_h(U).
\end{align*}
The finite-dimensional inf-sup theorem (\cite[Theorem 2.34]{ern2004}) therefore gives a solution $r_h\in X_h$ with \cref{eq:right-inverse}.
\end{proof}

\subsection{Construction of discrete CGO solutions}

We now construct the discrete CGO solutions. The strategy is to take a smooth cut-off harmonic exponential, project it onto the finite-element space via a shifted Galerkin problem, and then use \cref{lem:right-inverse} to correct the local discrete defect exactly. The resulting solutions will satisfy the discrete Schr\"odinger equation in $U$ and approximate the harmonic exponential on each fixed $U_0\Subset U$.
\begin{theorem}\label{thm:cgo}
	Fix an open set \(U_0\Subset U\). There exist constants \(s_0\geq1\) and \(c,C,h_0>0\), depending only on the fixed data and the domain separations, such that for every \(0<h\leq h_0\) and every \(\zeta\in\mathbb{C}^d\) satisfying \(\zeta\cdot\zeta=0\) and \(s:=|\operatorname{Re}\zeta|\in[s_0,ch^{-2/3}]\), there exists \(u_{\zeta,h}\in V_{h,0}(D)\) satisfying
	\begin{align}\label{eq:cgo-exact}
		a_{q,h}(u_{\zeta,h},v_h)=0,\qquad v_h\in Z_h(U).
	\end{align}
	For consistent mass, the solution takes the form \(u_{\zeta,h}=e^{\zeta\cdot x}(1+r_{\zeta,h})\) on \(U_0\), with the remainder bounded by
	\begin{align}\label{eq:cgo-consistent}
		\lVert r_{\zeta,h}\rVert_{L^2(U_0)}\leq C(s^{-1}+h^2s^3).
	\end{align}
	For lumped mass, the analogous bound holds at the vertices, with the nodal norm \(\lVert r_{\zeta,h}\rVert_{X_h,1,U_0}\) from \cref{eq:restricted-nodal-norm} replacing the \(L^2(U_0)\) norm. In both cases, we have
	\begin{align}\label{eq:cgo-H1}
		\lVert u_{\zeta,h}\rVert_{H^1(U_0)}\leq Ce^{Cs}.
	\end{align}
	The constants are uniform in the direction of \(\zeta\), in complex potentials bounded by \(M\), and in the mesh family.
\end{theorem}
\begin{proof}
	The construction proceeds in three steps: a shifted Galerkin approximation, an exact correction of the local defect, and the estimation of the relative error.
	
	\emph{Step 1: A shifted Galerkin approximation.}
	To construct a discrete function that locally resembles a harmonic exponential, we first use a cutoff. Choose $\eta\in C_c^\infty(D)$ equal to one near $\overline U$, and set $w=\eta e^{\zeta\cdot x}$. Then $w\in H_0^1(D)$, and $w=e^{\zeta\cdot x}$ is harmonic near $\overline U$.
	Set $\theta=e^{-\operatorname{Re}\zeta\cdot x}$ and $g=(-\Delta+\kappa s^2)w$. The identity $\zeta\cdot\zeta=0$ implies $|\operatorname{Im}\zeta|=s$, and hence
	\begin{align*}
		\theta g=e^{i\operatorname{Im}\zeta\cdot x}
		\bigl(\kappa s^2\eta-\Delta\eta-2\zeta\cdot\nabla\eta\bigr),
	\end{align*}
	so that $\lVert\theta g\rVert_{L^2(D)}\leq Cs^2$.
	Thus $w=(-\Delta_D+\kappa s^2)^{-1}g$. Let $w_h\in V_{h,0}(D)$ be its Galerkin approximation, defined by
	\begin{align*}
		(\nabla w_h,\nabla v_h)+\kappa s^2(w_h,v_h)=(g,v_h),\qquad v_h\in V_{h,0}(D).
	\end{align*}
	By \cref{lem:comparison,lem:interp}, we obtain the error bound
	\begin{align}\label{eq:seed-error}
		\lVert \theta(w-w_h)\rVert_{L^2(D)}\leq Ch^2s^2.
	\end{align}
	Since $\lVert\theta w\rVert_{L^2(D)}\leq C$ and $hs$ is bounded, \cref{eq:seed-error} also yields $\lVert\theta w_h\rVert_{L^2(D)}\leq C$.
	
	\emph{Step 2: Exact correction of the local discrete defect.}
	The Galerkin solution $w_h$ does not exactly solve the homogeneous local equation. To correct this, we first characterize its defect. For $v_h\in Z_h(U)$, Galerkin orthogonality and the harmonicity of $w$ near $\overline U$ give
	\begin{align*}
		(\nabla w_h,\nabla v_h)
		=(\nabla w,\nabla v_h)+\kappa s^2(w-w_h,v_h)
		=\kappa s^2(w-w_h,v_h).
	\end{align*}
	Adding the potential term yields the local defect
	\begin{align*}
		a_{q,h}(w_h,v_h)=\kappa s^2(w-w_h,v_h)+c_h(q;w_h,v_h).
	\end{align*}
	The first term is controlled by \cref{eq:seed-error}, and the second by $\lVert\theta w_h\rVert_{L^2(D)}\leq C$. Using \cref{eq:weighted-mass-equivalence} and the weighted Cauchy--Schwarz inequality, we obtain
	\begin{align*}
		|a_{q,h}(w_h,v_h)|\leq C(1+h^2s^4)\lVert v_h\rVert_{X_h,\theta^{-1}}.
	\end{align*}
	Applying \cref{lem:right-inverse} to the functional $F(v_h)=-a_{q,h}(w_h,v_h)$ yields a correction $z_h$ satisfying $a_{q,h}(z_h,v_h)=-a_{q,h}(w_h,v_h)$ for all $v_h\in Z_h(U)$, with the bound
	\begin{align}\label{eq:correction-bound}
		\lVert z_h\rVert_{X_h,\theta}\leq C(s^{-1}+h^2s^3).
	\end{align}
	Consequently, $u_{\zeta,h}=w_h+z_h$ exactly satisfies the local equation \cref{eq:cgo-exact}.
	
	\emph{Step 3: Relative and local energy bounds.}
	We now estimate the relative error on the inner set \(U_0\). On $U_0$, we have $w=e^{\zeta\cdot x}$, so we define
	\begin{align*}
		r_{\zeta,h}=e^{-\zeta\cdot x}\bigl(w_h-w+z_h\bigr).
	\end{align*}
	Since $|e^{-\zeta\cdot x}|=\theta$ and $s\geq1$, the bounds \cref{eq:seed-error,eq:correction-bound} immediately prove \cref{eq:cgo-consistent} for consistent mass.
	For lumped mass, we compare $w_h$ to the nodal interpolant $I_hw$, which satisfies $I_hw(x_i)=w(x_i)$ at each vertex. Smooth interpolation gives $\lVert\theta(I_hw-w)\rVert_{L^2(D)}\leq Ch^2s^2$, so \cref{eq:seed-error,eq:weighted-mass-equivalence} imply
	\begin{align*}
		\lVert w_h-I_hw\rVert_{X_h,\theta}
		\leq C\lVert\theta(w_h-I_hw)\rVert_{L^2(D)}
		\leq Ch^2s^2.
	\end{align*}
	Together with \cref{eq:correction-bound} and $|e^{-\zeta\cdot x_i}|=\theta(x_i)$, this proves the nodal remainder bound.
	Finally, to obtain the \(H^1\) estimate, choose $U_0\Subset G_1\Subset G\Subset U$. The preceding weighted bounds and $h^2s^3\leq C$ yield $\lVert u_{\zeta,h}\rVert_{L^2(G_1)}\leq Ce^{Cs}$ in either mass realization. Since $\mathcal V_h(G)\subset Z_h(U)$, the local exactness \cref{eq:cgo-exact} implies \cref{eq:local-equation} with $b=0$ and $f=0$. Applying \cref{lem:caccioppoli} then proves \cref{eq:cgo-H1}.
\end{proof}

Balancing $s^{-1}$ and $h^2s^3$ gives $s\asymp h^{-1/2}$ and a remainder of order $h^{1/2}$. This choice lies within both the linear range $s\leq ch^{-2/3}$ and the radial range $s\leq ch^{-4/5}$ for sufficiently small $h$. We use it in the inverse boundary estimates below.

\section{Quantitative discrete continuation}\label{sec:local}

In this section, we use the radial Carleman estimate \eqref{eq:radial-load} to derive quantitative continuation estimates for local variational equations. Our goal is to show that smallness on a boundary patch propagates to the interior and ultimately controls the full boundary flux. The argument proceeds in four steps. First, nodal cutoffs convert local equations into global ones driven by localized \(L^2\) loads. Second, the radial estimate yields a discrete three-ball inequality with an exponentially small mesh remainder. Third, we establish an interior propagation of smallness for dual residuals, which accommodates functionals lacking an \(L^2\) representative. Finally, we extend these results across an observed boundary face to obtain a boundary continuation estimate.

\subsection{Localization by nodal cutoffs}

To apply the global radial Carleman estimate to a function satisfying an equation only locally, we must encode its localized action as a global load. For a given cutoff $\chi$, let $\omega_h(\chi)$ denote the union of all closed mesh elements containing at least one vertex \(x_i\) with \(\chi(x_i)\ne0\). Then $\operatorname{supp}I_h(\chi v_h)\subset\omega_h(\chi)$ for any $v_h$.
The following lemma shows that the commutator produced by the cutoff can be represented by a localized \(L^2\) load \(F_\chi\). Crucially, the weighted norm of \(F_\chi\) can be estimated without inverting the global mass matrix, relying instead on the local support of the transition region.
\begin{lemma}\label{lem:cutoff-load}
	Let $\chi$ be a fixed smooth real cutoff, and suppose \cref{eq:local-equation} holds for all tests $I_h(\chi v_h)$ with $v_h\in V_{h,0}(D)$. Then there exists a broken affine function $F_\chi\in L^2(D)$ such that $z_h=I_h(\chi u_h)$ satisfies
	\begin{align}\label{eq:cutoff-equation}
		a_{b,q,h}(z_h,v_h)=(F_\chi,v_h)_{L^2(D)},\qquad v_h\in V_{h,0}(D).
	\end{align}
	If $f=0$ on $\omega_h(\chi)$, the load is supported in the transition region of $\chi$, and
	\begin{align}\label{eq:cutoff-estimate}
		\lVert \theta F_\chi\rVert_{L^2(D)}
		\leq C\bigl(\lVert \theta\nabla u_h\rVert_{L^2(\omega_{\rm tr}(\chi))}
		+\lVert \theta u_h\rVert_{L^2(\omega_{\rm tr}(\chi))}\bigr),\qquad hs\leq c,
	\end{align}
	where $\omega_{\rm tr}(\chi)$ is a fixed neighborhood containing these transition elements for sufficiently small $h$.
	The constants depend only on the fixed data and $\chi$.
	For nonzero $f$, add $C\lVert\theta f\rVert_{L^2(\omega_h(\chi))}$ to the right-hand side of \cref{eq:cutoff-estimate}.
	The assertion also holds on a fitted domain for a cutoff equal to one near its boundary, provided \(u_h\) has zero trace and \(f=0\) on $\omega_h(\chi)$.
\end{lemma}

\begin{proof}
	Set \(e_\chi(v)=I_h(\chi v)-\chi v\). Since \(\chi\) is real, testing \cref{eq:local-equation} with \(I_h(\chi v_h)\) and integrating by parts only on the Laplace term (so that no derivatives of \(b\) occur) yields
	\begin{align}\label{eq:commutator-exact}
		a_{b,q,h}(I_h(\chi u_h),v_h)
		&=(f,I_h(\chi v_h))
		+\bigl(-2\nabla\chi\cdot\nabla u_h-(\Delta\chi)u_h
		+(b\cdot\nabla\chi)u_h,v_h\bigr)\notag\\
		&\quad+(\nabla e_\chi(u_h),\nabla v_h)-(\nabla u_h,\nabla e_\chi(v_h))\notag\\
		&\quad+(b\cdot\nabla e_\chi(u_h),v_h)-(b\cdot\nabla u_h,e_\chi(v_h))\notag\\
		&\quad+c_h(q;I_h(\chi u_h),v_h)-c_h(q;u_h,I_h(\chi v_h)).
	\end{align}
	For consistent mass, the last line reduces to \(c_h(q;e_\chi(u_h),v_h)-c_h(q;u_h,e_\chi(v_h))\); for lumped mass, it vanishes identically because the potential operator commutes with nodal multiplication. 
	To bound the interpolation errors, we use the standard product estimate
	\begin{align}\label{eq:cutoff-product}
		\lVert e_\chi(w_h)\rVert_{L^2(K)}+h_K\lVert\nabla e_\chi(w_h)\rVert_{L^2(K)}
		\leq Ch_K^2\bigl(\lVert\nabla w_h\rVert_{L^2(K)}+\lVert w_h\rVert_{L^2(K)}\bigr),
	\end{align}
	together with the inverse estimate for affine functions. This allows us to bound the two stiffness interpolation terms in \cref{eq:commutator-exact} elementwise by \(C(\lVert\nabla u_h\rVert_{L^2(K)}+\lVert u_h\rVert_{L^2(K)})\lVert v_h\rVert_{L^2(K)}\). Similarly, the two first-order interpolation terms satisfy
	\begin{align*}
		|(b\cdot\nabla e_\chi(u_h),v_h)_{L^2(K)}|
		+|(b\cdot\nabla u_h,e_\chi(v_h))_{L^2(K)}|
		\leq CM_1h_K\bigl(\lVert\nabla u_h\rVert_{L^2(K)}+\lVert u_h\rVert_{L^2(K)}\bigr)
		\lVert v_h\rVert_{L^2(K)}.
	\end{align*}
	For \(h\leq1\), the potential errors and the explicit commutator in \cref{eq:commutator-exact} obey the same bound without the factor \(h_K\), while the forcing contributes at most \(C\lVert f\rVert_{L^2(K)}\lVert v_h\rVert_{L^2(K)}\) on elements within \(\omega_h(\chi)\). Since every commutator contribution vanishes where \(\chi\) is constant, the remaining terms are supported only in the transition region. Representing each element functional by its \(L^2(K)\) Riesz representative in \(P_1(K)\) and assembling them, we obtain the load \(F_\chi\) satisfying \cref{eq:cutoff-equation}. Finally, freezing the weight elementwise and using \(hs\leq c\) proves the estimate \cref{eq:cutoff-estimate}. The boundary cutoff case follows identically, using the zero-trace property of the tests.
\end{proof}

\subsection{Propagation of smallness}

Using the localized load, we now derive a three-ball inequality. The key algebraic ingredient is the truncated parameter optimization from \cite[Section~4.1]{FBRRS}, which balances the exponential growth and decay terms in the Carleman estimate.
\begin{lemma}\label{lem:truncated-optimization}
	Let $a,b,C_0>0$, $S\geq s_0\geq1$, and $x,D,M\geq0$. Suppose $x\leq C_0M$ and
	\begin{align*}
		x\leq C_0(e^{as}D+e^{-bs}M),\qquad s_0\leq s\leq S.
	\end{align*}
	Then, with $\vartheta=b/(a+b)$,
	\begin{align*}
		x\leq C D^\vartheta M^{1-\vartheta}+C e^{-bS}M.
	\end{align*}
	The constant $C$ depends only on $a,b,s_0,C_0$. Moreover, any finite fixed composition of such estimates retains the same form, with a positive exponent and a positive exponential constant.
\end{lemma}
\begin{proof}
	For \(0 < D < M\), we choose \(s = \min\{S, \max\{s_0, (a+b)^{-1}\log(M/D)\}\}\). If the minimum is attained at the unconstrained optimizer \((a+b)^{-1}\log(M/D)\), substituting this value yields \(e^{as}D = e^{-bs}M = D^\vartheta M^{1-\vartheta}\), which directly produces the interpolation term. If the choice is truncated at the upper bound \(S\), we have \(e^{-bS}M \geq D^\vartheta M^{1-\vartheta}\), so the bound is controlled by \(C_0 e^{-bS}M\). If it is truncated at the lower bound \(s_0\), the condition \((a+b)^{-1}\log(M/D) \leq s_0\) implies \(D \geq M\), which is covered by the a priori estimate \(x \leq C_0M \leq C_0 D^\vartheta M^{1-\vartheta}\).
	
	The remaining boundary cases follow immediately. For \(D=0\), taking \(s=S\) directly yields \(x \leq C_0 e^{-bS}M\). For \(D \geq M\), the a priori bound \(x \leq C_0M\) already implies the desired inequality since \(M \leq D^\vartheta M^{1-\vartheta}\). Finally, the preservation of the bound under finite composition follows from the elementary subadditivity \((x+y)^\vartheta \leq x^\vartheta + y^\vartheta\) for \(0 < \vartheta < 1\).
\end{proof}

Combining \cref{eq:cutoff-equation,eq:cutoff-estimate,eq:radial-load} with the optimization lemma yields the following propagation of smallness, up to an exponentially small mesh remainder.
\begin{proposition}\label{prop:threeball}
	Let $B_1\Subset B_2\Subset B_3\Subset G$ be fixed concentric balls. There exist constants \(s_0\geq1\), \(h_0,c_+,c_-,c,C>0\), and \(\vartheta\in(0,1)\), depending only on the fixed data and the three balls, such that for \(0<h\leq h_0\) and \(s_0\leq s\leq ch^{-4/5}\), every solution of \cref{eq:local-equation} satisfies
	\begin{align}\label{eq:threeball-param}
		\lVert u_h\rVert_{L^2(B_2)}
		&\leq Ce^{c_+s}\bigl(\lVert u_h\rVert_{L^2(B_1)}+\lVert f\rVert_{L^2(B_3)}\bigr)
		+Ce^{-c_-s}\bigl(\lVert u_h\rVert_{L^2(B_3)}+\lVert f\rVert_{L^2(B_3)}\bigr).
	\end{align}
	Optimizing the parameter \(s\) further yields the interpolation estimate
	\begin{align}\label{eq:threeball}
		\lVert u_h\rVert_{L^2(B_2)}
		&\leq C\bigl(\lVert u_h\rVert_{L^2(B_1)}+\lVert f\rVert_{L^2(B_3)}\bigr)^\vartheta
		\bigl(\lVert u_h\rVert_{L^2(B_3)}+\lVert f\rVert_{L^2(B_3)}\bigr)^{1-\vartheta}\notag\\
		&\quad+Ce^{-ch^{-4/5}}\bigl(\lVert u_h\rVert_{L^2(B_3)}+\lVert f\rVert_{L^2(B_3)}\bigr).
	\end{align}
	The same estimate holds along any finite fixed chain of interior balls, with adjusted constants and exponent.
\end{proposition}

\begin{proof}
	Choose a radial cutoff \(\chi\) equal to one near \(\overline B_2\) and supported strictly inside \(B_3\), with its transition region separated from \(\overline B_2\). For sufficiently small \(h\), we have \(\omega_h(\chi)\Subset B_3\) and can choose \(\omega_{\rm tr}(\chi)\Subset B_3\setminus\overline B_2\). Center the phase \cref{eq:radial-weight} at the common center of the balls and choose its observation ball strictly inside \(B_1\). Define \(c_+=\sup_{B_3}\rho-\inf_{B_2}\rho>0\). Since \(\rho\) decreases radially, the outer transition satisfies \(c_-=\inf_{B_2}\rho-\sup_{\omega_{\rm tr}(\chi)}\rho>0\).
	
	Applying \cref{eq:radial-load} to \(I_h(\chi u_h)\) with the load from \cref{eq:cutoff-equation}, taking square roots, dividing by \(s e^{s\inf_{B_2}\rho}\), and using \cref{eq:cutoff-estimate} yields
	\begin{align*}
		\lVert u_h\rVert_{L^2(B_2)}
		&\leq Ce^{c_+s}\bigl(\lVert u_h\rVert_{L^2(B_1)}+s^{-3/2}\lVert f\rVert_{L^2(B_3)}\bigr)
		+Cs^{-3/2}e^{-c_-s}\lVert u_h\rVert_{H^1(\omega_{\rm tr}(\chi))}.
	\end{align*}
	By \cref{eq:caccioppoli} on a fixed larger annulus compactly contained in \(B_3\) and the fact that \(s\geq1\), the boundary term can be absorbed, giving the parameterized estimate \cref{eq:threeball-param}. Finally, applying \cref{lem:truncated-optimization} to \cref{eq:threeball-param} with \(S=ch^{-4/5}\) and the a priori bound \(\lVert u_h\rVert_{L^2(B_2)}\leq\lVert u_h\rVert_{L^2(B_3)}\) yields \cref{eq:threeball} with exponent \(\vartheta=c_-/(c_++c_-)\). The finite-composition assertion of the lemma directly extends this result to a fixed chain of balls.
\end{proof}

\subsection{Interior continuation with a dual residual}

In the subsequent boundary continuation argument, the local equations are driven by boundary flux functionals rather than \(L^2\) functions. To accommodate this, we establish a propagation of smallness for dual residuals. For an antilinear functional $F_h$ on $\mathcal V_h(G)$, define
\begin{align*}
	\lVert F_h\rVert_{-1,h;G}
	=\sup_{0\ne v_h\in\mathcal V_h(G)}
	\frac{|F_h(v_h)|}{\lVert v_h\rVert_{H^1(D)}},
\end{align*}
with the supremum set to zero if the test space is trivial. The proof reduces the residual to an \(L^2\) forcing via a coercive correction, enabling the application of the three-ball inequality.
\begin{theorem}\label{thm:interior-quc}
	Let $G\subset D$ be connected and open, and let $\omega,K\Subset G$ be nonempty open sets. Suppose $u_h\in V_h(D)$ and an antilinear functional $F_h$ on $\mathcal V_h(G)$ satisfy $a_{b,q,h}(u_h,v_h)=F_h(v_h)$ for all \(v_h\in\mathcal V_h(G)\). Set $D_h=\lVert u_h\rVert_{L^2(\omega)}+\lVert F_h\rVert_{-1,h;G}$ and $M_h=\lVert u_h\rVert_{L^2(G)}+\lVert F_h\rVert_{-1,h;G}$. There exist constants \(C,c,h_0>0\) and \(\vartheta\in(0,1)\), depending only on the fixed data and these sets, such that for all \(0<h\leq h_0\), we have
	\begin{align}\label{eq:interior-quc}
		\lVert u_h\rVert_{H^1(K)}
		\leq C D_h^\vartheta M_h^{1-\vartheta}+Ce^{-ch^{-4/5}}M_h.
	\end{align}
\end{theorem}
\begin{proof}
	By the Hahn--Banach theorem, extend \(F_h\) to \(V_{h,0}(D)\) without increasing its \(H^1\) dual norm. For either mass realization, Young's inequality and \cref{eq:weighted-mass-equivalence} with $\sigma=1$ yield the coercivity bound
	\begin{align*}
		\operatorname{Re}a_{b,q,h}(v_h,v_h)
		\geq\frac12\lVert\nabla v_h\rVert_{L^2(D)}^2-C(M+M_1^2)\lVert v_h\rVert_{L^2(D)}^2.
	\end{align*}
	We then introduce a coercive correction by fixing \(\mu_0>C(M+M_1^2)+1\) and solving the shifted problem
	\begin{align*}
		a_{b,q,h}(z_h,v_h)+\mu_0(z_h,v_h)_{L^2(D)}=F_h(v_h),
		\qquad v_h\in V_{h,0}(D).
	\end{align*}
	This problem is uniformly coercive in \(H^1(D)\) and admits a unique solution \(z_h\in V_{h,0}(D)\) satisfying \(\lVert z_h\rVert_{H^1(D)}\leq C\lVert F_h\rVert_{-1,h;G}\).
	
	Subtracting this equation from \(a_{b,q,h}(u_h,v_h)=F_h(v_h)\) yields a local equation for the difference:
	\begin{align*}
		a_{b,q,h}(u_h-z_h,v_h)=\mu_0(z_h,v_h)_{L^2(D)},\qquad v_h\in\mathcal V_h(G),
	\end{align*}
	where the right-hand side is an \(L^2\) forcing bounded by \(\lVert\mu_0z_h\rVert_{L^2(D)}\leq C\lVert F_h\rVert_{-1,h;G}\). 
	To propagate this estimate, choose a ball compactly contained in \(\omega\). The connectedness of \(G\) and compactness of \(\overline K\) ensure there exist finitely many fixed interior ball chains connecting this ball to a neighborhood of \(\overline K\). Applying \cref{eq:threeball} along these chains and \cref{eq:caccioppoli} on the terminal balls gives the estimate \cref{eq:interior-quc} for \(u_h-z_h\). Finally, adding back the correction \(z_h\) preserves the bound, since its \(H^1(K)\) norm is controlled by \(C\lVert F_h\rVert_{-1,h;G}\leq CD_h\leq CD_h^\vartheta M_h^{1-\vartheta}\), completing the proof.
\end{proof}

\subsection{Variational boundary continuation}\label{sec:boundary}\label{sec:models}

We now extend the interior propagation to the boundary. This reduction follows the continuous Cauchy stability approach; see \cite{ARRV2009}. We implement it on the fitted mesh, representing the residual and flux as functionals. The trace norm introduced below will also feature in the data discrepancy later.
Throughout this subsection, let \(d\geq2\) and let \(\Omega\subset\mathbb{R}^d\) be a fixed bounded connected Lipschitz polyhedral domain. We assume
\(\overline\Omega\subset U\Subset D\), with fixed positive separations, where \(D\) is a box.
The triangulations of \(D\) are uniformly shape regular, fit \(\partial\Omega\), and have maximum diameter \(h\); their restrictions define the spaces on \(\Omega\).
We use either mass realization from \cref{sec:mass}, with \(a_{b,q,h}\) from \cref{eq:drift-form} assembled over \(\Omega\), including boundary vertices in the mass and potential forms.
Potentials are complex and bounded in modulus by \(M\), while the drift \(b\in L^\infty(\Omega;\mathbb{C}^d)\) satisfies \(\lVert b\rVert_{L^\infty(\Omega)}\leq M_1\).
When needed on \(D\), we extend \(b\) as a bounded measurable vector field, and extend the potentials as measurable functions or nodal values, with the same bounds.

Let \(T_h=\operatorname{tr}V_h(\Omega)\) and equip it with the quotient norm
\begin{align}\label{eq:trace-norm}
	\lVert g\rVert_{T_h}=\inf_{\substack{v_h\in V_h(\Omega)\\\operatorname{tr} v_h=g}}\lVert v_h\rVert_{H^1(\Omega)}.
\end{align}
The trace theorem gives $\lVert g\rVert_{H^{1/2}(\partial\Omega)}\leq C\lVert g\rVert_{T_h}$. For the reverse inequality, apply a boundary-preserving Scott--Zhang projection \cite[Section~5]{SZ} to an \(H^1\)-lifting of \(g\) bounded by \(C\lVert g\rVert_{H^{1/2}(\partial\Omega)}\). Since \(g\in T_h\), the projected function retains the trace \(g\), and the \(H^1\) stability of the projection yields $\lVert g\rVert_{T_h}\leq C\lVert g\rVert_{H^{1/2}(\partial\Omega)}$, uniformly in \(h\).
For a relatively open \(\Gamma\subset\partial\Omega\), let \(T_{h,\Gamma}\subset T_h\) consist of traces vanishing on \(\partial\Omega\setminus\Gamma\). Primes denote continuous antilinear duals, with \(V_h(\Omega)\) equipped with the \(H^1(\Omega)\) norm and \(T_{h,\Gamma}\) with the inherited \(T_h\) norm.

The following theorem bounds the interior \(H^1\) norm by the dual residual and the flux on the observed patch.

\begin{theorem}\label{thm:boundary-quc}
	Assume the fitted geometry and mass realization of \cref{sec:models} with \(d\geq2\). Let \(\Gamma\subset\partial\Omega\) contain a fixed nonempty patch in the interior of a face. Fix a ball \(B\Subset D\) centered on that face, meeting \(\partial\Omega\) only in the interior of that face, with \(B\cap\partial\Omega\Subset\Gamma\). Let \(G\subset\Omega\) be connected and open, containing \(B\cap\Omega\), and let \(K\Subset G\) be open.
	Suppose \(u_h\in V_h(\Omega)\) has zero trace on \(\Gamma\), and that \(F_h\in V_h(\Omega)'\) and \(\beta_h\in T_{h,\Gamma}^\prime\) satisfy
	\begin{align}\label{eq:cauchy-interior-balance}
		a_{b,q,h}(u_h,v_h)=F_h(v_h),\qquad \operatorname{supp}v_h\Subset G.
	\end{align}
	For tests meeting the observed face, assume
	\begin{align}\label{eq:cauchy-local-balance}
		a_{b,q,h}(u_h,v_h)=F_h(v_h)+\langle\beta_h,\operatorname{tr}v_h\rangle,
		\qquad \operatorname{supp}v_h\Subset B\cap\overline\Omega,
	\end{align}
	where support in the latter identity is taken relative to \(\overline\Omega\). Set $D_h=\lVert F_h\rVert_{V_h(\Omega)'}+\lVert \beta_h\rVert_{T_{h,\Gamma}^\prime}$ and $M_h=\lVert u_h\rVert_{L^2(\Omega)}+D_h$.
	There exist constants \(C,c,h_0>0\) and \(\vartheta\in(0,1)\), depending only on the fixed data and these sets, such that for all \(0<h\leq h_0\),
	\begin{align}\label{eq:boundary-quc}
		\lVert u_h\rVert_{H^1(K)}\leq C D_h^\vartheta M_h^{1-\vartheta}+Ce^{-ch^{-4/5}}M_h.
	\end{align}
\end{theorem}
\begin{proof}
	\emph{Step 1: Nodal extension and the extended residual.}
	Choose a smaller concentric ball \(B'\Subset B\) and a cutoff $\chi\in C_c^\infty(B)$ equal to one near \(\overline{B'}\). Define \(u_h^0\in V_{h,0}(D)\) by the nodal values of \(u_h\) on \(\overline\Omega\) and zero elsewhere. Then \(u_h^0\) agrees with \(u_h\) in \(\Omega\), and since \(\operatorname{tr} u_h=0\) on \(\Gamma\), it vanishes on $B'\setminus\overline\Omega$ for sufficiently small \(h\).
	Extending the potential and \(b\) boundedly to \(D\), we define the extended residual functional
	\begin{align*}
		\mathfrak{f}_h(v_h)=F_h(v_h|_\Omega)
		+\langle\beta_h,\operatorname{tr}_\Omega I_h(\chi v_h)\rangle,
		\qquad v_h\in V_{h,0}(D).
	\end{align*}
	For small \(h\), \(\operatorname{tr}_\Omega I_h(\chi v_h)\in T_{h,\Gamma}\). By \cref{eq:trace-norm,eq:cutoff-product}, it satisfies the uniform bound
	\begin{align*}
		|\mathfrak{f}_h(v_h)|\leq CD_h\lVert v_h\rVert_{H^1(D)}.
	\end{align*}
	
	\emph{Step 2: The extended local equation.}
	We extend the local equation across the boundary. Choose a fixed exterior ball \(\omega\Subset B'\setminus\overline\Omega\) and a connected open set \(\widetilde G\Subset G\cup B'\) containing \(\overline\omega\cup\overline K\). The open cover \(\{G,B'\}\) of \(\overline{\widetilde G}\) has a positive Lebesgue number, so we may decrease \(h_0\) to ensure every nodal star intersecting \(\widetilde G\) lies entirely in either \(G\) or \(B'\).
	Decompose a test function supported in \(\widetilde G\) into nodal basis functions. For a star in \(G\), apply \cref{eq:cauchy-interior-balance} (its trace contribution vanishes). For a star in \(B'\), apply \cref{eq:cauchy-local-balance} to its restriction to \(\Omega\), using \(\chi=1\) near \(B'\).
	The stiffness, drift, and consistent potential integrals agree with those on \(\Omega\): the mesh is fitted, and both \(u_h^0\) and its gradient vanish on the exterior elements near the observed face. For lumped mass, the differing masses at face vertices do not contribute since \(u_h=0\) there. Summing these nodal identities yields
	\begin{align*}
		a_{b,q,h}(u_h^0,v_h)=\mathfrak{f}_h(v_h),
		\qquad v_h\in\mathcal V_h(\widetilde G).
	\end{align*}
	
	\emph{Step 3: Application of interior continuation.}
	We now apply the interior result. We have $\lVert u_h^0\rVert_{L^2(\widetilde G)}\leq\lVert u_h\rVert_{L^2(\Omega)}$ and $u_h^0=0$ on \(\omega\). Since \(\lVert\mathfrak f_h\rVert_{-1,h;\widetilde G}\leq CD_h\), applying the interior continuation theorem \cref{thm:interior-quc} directly yields the desired boundary estimate \cref{eq:boundary-quc}.
\end{proof}

For nonzero Dirichlet data, define the restriction space
$\widehat T_h(\Gamma)=\{\operatorname{tr}v_h|_\Gamma:v_h\in V_h(\Omega)\}$ equipped with the quotient norm
\begin{align*}
	\lVert g\rVert_{\widehat T_h(\Gamma)}
	=\inf_{\substack{v_h\in V_h(\Omega)\\\operatorname{tr}v_h|_\Gamma=g}}
	\lVert v_h\rVert_{H^1(\Omega)}.
\end{align*}
Note that this space differs from $T_{h,\Gamma}$, whose elements vanish outside \(\Gamma\).
Lifting this trace reduces the next estimate to \cref{thm:boundary-quc}.
\begin{corollary}\label{cor:cauchy-data}
	In the setting of \cref{thm:boundary-quc}, we drop the zero-trace assumption on \(u_h\) and suppose \cref{eq:cauchy-interior-balance,eq:cauchy-local-balance} hold. Then \cref{eq:boundary-quc} holds with
	\begin{align*}
		D_h=\lVert \operatorname{tr}u_h|_\Gamma\rVert_{\widehat T_h(\Gamma)}
		+\lVert F_h\rVert_{V_h(\Omega)'}+\lVert \beta_h\rVert_{T_{h,\Gamma}^\prime},
	\end{align*}
	and $M_h=\lVert u_h\rVert_{L^2(\Omega)}+D_h$.
\end{corollary}
\begin{proof}
Let \(e_h\) be a minimum-norm lifting of the measured Dirichlet trace. The difference \(u_h-e_h\) has zero trace on \(\Gamma\) and satisfies \cref{eq:cauchy-interior-balance,eq:cauchy-local-balance} with the same \(\beta_h\) and a modified residual \(\widetilde F_h=F_h-a_{b,q,h}(e_h,\cdot)\). By the continuity of the form \(a_{b,q,h}\), this new residual satisfies
\begin{align*}
	\lVert \widetilde F_h\rVert_{V_h(\Omega)'}\leq\lVert F_h\rVert_{V_h(\Omega)'}+C\lVert e_h\rVert_{H^1(\Omega)}\leq CD_h.
\end{align*}
Applying \cref{thm:boundary-quc} to \(u_h-e_h\) and then adding back \(e_h\) (using \(D_h\leq M_h\)) yields the desired estimate.
\end{proof}

\section{Calder\'on stability from two boundary patches}\label{sec:inverse}

In this section, we establish our main inverse stability result. Throughout this section, let \(d\geq3\) and set \(b=0\). We adopt the geometry, potentials, and trace spaces of \cref{sec:models}, with $m_h$ and $a_{q,h}$ defined as in \cref{eq:consistent,eq:lumped} and \cref{eq:aq}, respectively. Our strategy combines the discrete CGO solutions, the full-data Fourier estimate, and the boundary continuation estimates developed above. We proceed in four steps:
\begin{enumerate}
	\item We first define the variational Dirichlet-to-Neumann map and derive a full-data Fourier estimate via the CGO solutions and the Alessandrini identity.
	\item Second, we convert this low-frequency control into a full-data \(H^{-r}\) stability estimate via a standard frequency-splitting argument.
	\item Third, using reciprocity and the boundary continuation estimates, we localize both Dirichlet inputs and flux observations to independently prescribed, possibly disjoint patches.
	\item Finally, we show that for coefficient differences in a fixed finite-dimensional space, the mesh remainder can be absorbed, yielding mesh-uniform Lipschitz stability.
\end{enumerate}

\subsection{Variational DtN maps and coefficient realization}
We begin by defining the discrete forward problem and its associated boundary map. Let \(T_h=\operatorname{tr}V_h(\Omega)\) be the trace space equipped with the quotient norm \eqref{eq:trace-norm}. Unique solvability is essential for the boundary map to be well-defined. Furthermore, boundary localization requires uniform stability with respect to the mesh and the potential. We encode these requirements in the following assumption.
\begin{assumption}\label{ass:forward}
	For every potential under consideration, the only $v\in V_{h,0}(\Omega)$ satisfying $a_{q,h}(v,w)=0$ for all $w\in V_{h,0}(\Omega)$ is $v=0$.
	If boundary localization is used, we additionally assume
	\begin{align}\label{eq:dir-infsup}
		\lVert v\rVert_{H^1(\Omega)}\leq C_D
		\sup_{0\ne w\in V_{h,0}(\Omega)}
		\frac{|a_{q,h}(v,w)|}{\lVert w\rVert_{H^1(\Omega)}},\qquad v\in V_{h,0}(\Omega),
	\end{align}
	with $C_D$ independent of $h$ and the potential.
\end{assumption}
Since the discrete Dirichlet matrix is square, the uniqueness condition is equivalent to its invertibility, and hence to unique solvability for every boundary trace.
Potentials with $\operatorname{Re}q\geq0$ satisfy \cref{eq:dir-infsup} by Poincar\'e's inequality.

For \(g\in T_h\), let \(u_h[q,g]\) solve
\begin{align}\label{eq:forward}
	\operatorname{tr} u_h=g,\qquad a_{q,h}(u_h,v_h)=0\quad(v_h\in V_{h,0}(\Omega)).
\end{align}
The associated variational Dirichlet-to-Neumann map is defined by
\begin{align}\label{eq:dtn}
	\langle \Lambda_h(q)g,k\rangle=a_{q,h}(u_h[q,g],v_h),\qquad k\in T_h,
\end{align}
where $v_h\in V_h(\Omega)$ is any lifting with trace \(k\). By \cref{eq:forward}, this definition is independent of the lifting.
For a relatively open \(\Gamma\subset\partial\Omega\), define the boundary data discrepancy
\begin{align*}
	\delta_\Gamma(q_1,q_2)=
	\sup_{\substack{g\in T_h,\ k\in T_{h,\Gamma}\\
			\lVert g\rVert_{T_h}=\lVert k\rVert_{T_h}=1}}
	\bigl|\langle [\Lambda_h(q_1)-\Lambda_h(q_2)]g,k\rangle\bigr|.
\end{align*}
For a fixed pair $q_1,q_2$, we abbreviate $\delta_\Gamma=\delta_\Gamma(q_1,q_2)$ and $\delta_\partial=\delta_{\partial\Omega}(q_1,q_2)$. Thus, Dirichlet inputs are allowed on the whole boundary, while flux observations are restricted to \(\Gamma\).
In nodal coordinates, $\Lambda_h(q)$ is the Schur complement of the interior block of the assembled matrix for $a_{q,h}$, so the boundary data match those of the original discrete scheme.

To compare potentials as functions, we define the discrete realization \(\mathcal{Q}_h\). For consistent mass, \(\mathcal{Q}_hq\) is \(q\) on \(\Omega\), extended by zero. For lumped mass, we choose a measurable nodal partition \(\{D_i\}\) of \(\Omega\) with \(|D_i|=m_i\) and \(D_i\subset B(x_i,Ch)\), and set \(\mathcal{Q}_hq=q_i\) on \(D_i\) (and zero outside); barycentric dual cells provide such a partition.
For a Cartesian grid with zero boundary nodal potentials, centered interior dual cubes may be used instead, with zero value on the remaining boundary strip.
In either case, \(\lVert \mathcal{Q}_hq\rVert_{L^\infty(\mathbb{R}^d)}\leq M\) and the support lies in a fixed bounded set.

We note a fundamental limitation of the consistent mass realization. In that case, the discrete operator depends on \(q\) only through the finitely many moments \(\int_\Omega q\phi_i\phi_j\,dx\), where \(\{\phi_i\}\) is the nodal basis of \(V_h(\Omega)\). Since the map from \(L^\infty(\Omega)\) to these moments has an infinite-dimensional kernel, fixed-mesh boundary data cannot uniquely determine an arbitrary bounded potential. Consequently, the estimates below for unrestricted bounded potentials retain a mesh remainder even for identical data. For lumped mass, \(\mathcal{Q}_hq\) instead represents the finite-dimensional nodal potential.

All negative Sobolev norms below are taken on \(\mathbb{R}^d\) after this zero extension.

\subsection{Alessandrini identity and CGO Fourier probing}

We now use the CGO solutions from \cref{sec:cgo} to probe the Fourier transform of the potential difference. The key tool is the variational Alessandrini identity, which links the boundary discrepancy to an interior integral involving the potential difference and the two solutions.
\begin{lemma}\label{lem:alessandrini}
	Assume \cref{ass:forward} and let $u_j=u_h[q_j,g_j]$. Then
	\begin{align}\label{eq:alessandrini}
		\langle [\Lambda_h(q_1)-\Lambda_h(q_2)]g_1,\overline{g_2}\rangle
		=c_h(q_1-q_2;u_1,\overline{u_2}).
	\end{align}
	If \cref{eq:dir-infsup} holds, there exists a constant \(C>0\), depending only on the fixed data and \(C_D\), such that
	\begin{align}\label{eq:forward-bound}
		\lVert u_h[q,g]\rVert_{H^1(\Omega)}\leq C\lVert g\rVert_{T_h}
	\end{align}
	and $\lVert \Lambda_h(q)\rVert_{T_h\to T_h^\prime}\leq C$.
\end{lemma}
\begin{proof}
	We first establish the Alessandrini identity using the bilinear symmetry of the forms. Observe that the potential and stiffness forms satisfy
	\begin{align}\label{eq:bilinear-symmetry}
		a_{q,h}(u,\overline v)=a_{q,h}(v,\overline u).
	\end{align}
	By the definition of the DtN map in \cref{eq:dtn}, we can use $\overline{u_2}$ as a lifting of $\overline{g_2}$, which gives
	$\langle \Lambda_h(q_1)g_1,\overline{g_2}\rangle = a_{q_1,h}(u_1,\overline{u_2})$.
	To evaluate the corresponding term for $q_2$, set $w_h=u_1-u_h[q_2,g_1]$, which has zero trace. By the symmetry \cref{eq:bilinear-symmetry} and the fact that \(u_2\) is a solution for \(q_2\), we have
	\begin{align*}
		a_{q_2,h}(w_h,\overline{u_2})=a_{q_2,h}(u_2,\overline{w_h})=0.
	\end{align*}
	This implies $a_{q_2,h}(u_h[q_2,g_1],\overline{u_2}) = a_{q_2,h}(u_1,\overline{u_2})$. 
	Subtracting the pairings for $q_1$ and $q_2$, and using \cref{eq:bilinear-symmetry} on both terms, we obtain
	\begin{align*}
		\langle [\Lambda_h(q_1)-\Lambda_h(q_2)]g_1,\overline{g_2}\rangle
		= a_{q_1,h}(u_2,\overline{u_1})-a_{q_2,h}(u_2,\overline{u_1})
		= c_h(q_1-q_2;u_1,\overline{u_2}),
	\end{align*}
	which proves \cref{eq:alessandrini}. The conjugation on the second trace is precisely what converts the sesquilinear identity into the bilinear product required for Fourier probing, whose right-hand side reads $\int_\Omega(q_1-q_2)u_1u_2\,dx$ or $\sum_i m_i(q_{1,i}-q_{2,i})u_{1,i}u_{2,i}$ depending on the mass realization.
	
	To prove the forward bounds \cref{eq:forward-bound}, we use a standard lifting argument. Subtracting a minimum-norm lifting in \cref{eq:forward} and applying the inf-sup condition \cref{eq:dir-infsup} directly yields the \(H^1\) estimate. Substituting this estimate into \cref{eq:dtn} and using the boundedness of $a_{q,h}$ then gives the operator norm bound $\lVert\Lambda_h(q)\rVert_{T_h\to T_h'}\leq C$.
\end{proof}
We use the Fourier transform $\widehat Q(\xi)=\int_{\mathbb{R}^d}Q(x)e^{i\xi\cdot x}\,dx$, with normalization constants absorbed into \(C\).
Since \(\overline\Omega\subset U\), the zero extension of any test function in \(V_{h,0}(\Omega)\) belongs to \(Z_h(U)\). Elementwise assembly therefore shows that the restriction of a CGO solution to \(\Omega\) satisfies \cref{eq:forward}. By uniqueness, it coincides with the Dirichlet solution corresponding to its boundary trace. According to \cref{eq:cgo-H1,eq:trace-norm}, this trace has norm at most \(Ce^{Cs}\). Substituting two such solutions into the Alessandrini identity \cref{eq:alessandrini} yields control over the low-frequency Fourier transform of the potential difference.

\begin{lemma}\label{lem:fourier-full}
	For $Q=\mathcal{Q}_h(q_1-q_2)$, set $P=\lVert Q\rVert_{L^\infty(\mathbb{R}^d)}$. There exist constants $s_0,c,C,h_0>0$ such that for all $0<h\leq h_0$, $s_0\leq s\leq ch^{-1/2}$, and \(|\xi|\leq2s\), we have
	\begin{align}\label{eq:fourier-full}
		|\widehat Q(\xi)|\leq
		C\bigl(\delta_\partial e^{Cs}+P(s^{-1}+h^2s^3+hs)\bigr).
	\end{align}
	The term $hs$ is absent for consistent mass.
\end{lemma}

\begin{proof}
	Choose orthonormal vectors \(\mathbf{a}, \mathbf{b} \in \mathbb{R}^d\) perpendicular to \(\xi\) (possible since \(d \geq 3\)), and set
	\begin{align*}
		\zeta_1&=s\mathbf{a}+i\bigl(\xi/2+\sqrt{s^2-|\xi|^2/4}\,\mathbf{b}\bigr), \\
		\zeta_2&=-s\mathbf{a}+i\bigl(\xi/2-\sqrt{s^2-|\xi|^2/4}\,\mathbf{b}\bigr).
	\end{align*}
	At \(\xi=0\), any orthonormal pair may be chosen. These vectors satisfy \(\zeta_j\cdot\zeta_j=0\) and \(\zeta_1+\zeta_2=i\xi\).
	Substituting the corresponding CGO solutions from \cref{thm:cgo} into the Alessandrini identity \cref{eq:alessandrini}, the boundary contribution is immediately bounded by \(C\delta_\partial e^{Cs}\).
	
	For consistent mass, write \(u_j=e^{\zeta_j\cdot x}(1+r_j)\) with remainders \(r_j=r_{\zeta_j,h}\). Since \(u_1u_2=e^{i\xi\cdot x}(1+r_1+r_2+r_1r_2)\), the error terms in \cref{eq:alessandrini} involve \((q_1-q_2)r_1\), \((q_1-q_2)r_2\), and \((q_1-q_2)r_1r_2\). By the bound \cref{eq:cgo-consistent} and the Cauchy--Schwarz inequality, their total contribution is bounded by \(CP(s^{-1}+h^2s^3)\), where we used \(s^{-1}+h^2s^3\leq C\) to control the product term \(r_1r_2\).
	For lumped mass, the same calculation applies to the nodal remainders from \cref{thm:cgo}. In this case, the difference between the discrete sum \(\sum_i m_i(q_{1,i}-q_{2,i})e^{i\xi\cdot x_i}\) and the true Fourier transform \(\widehat Q(\xi)\) is bounded by
	\begin{align*}
		P\sum_i\int_{D_i}|e^{i\xi\cdot x}-e^{i\xi\cdot x_i}|\,dx
		\leq CPh|\xi|\leq CPhs.
	\end{align*}
	Combining these estimates proves the asserted bound \cref{eq:fourier-full}.
\end{proof}

\subsection{Full-data stability}

Building on this low-frequency estimate, we now establish the full-data stability bound in negative Sobolev spaces. The stability modulus depends on the Sobolev exponent \(r\) and is described by
\begin{align*}
	\Phi_r(t)=
	\begin{cases}
		t^{2r/d},&0<r<d/2,\\
		t\sqrt{\log(2+t^{-1})},&r=d/2,\\
		t,&r>d/2,
	\end{cases}
	\qquad t>0.
\end{align*}
Here and below, we interpret $[\log(2+\delta^{-1})]^{-1}$ as zero when $\delta=0$.
The following full-data estimate provides the coefficient bound required for the boundary localization in \cref{thm:patch}.

\begin{theorem}\label{thm:full}
	Assume the geometric hypotheses of \cref{sec:models} and the uniqueness condition in \cref{ass:forward}. For every $r>0$, there exist constants \(C,h_0>0\), depending only on the fixed geometric sets, \(d,r,M\), and the mesh regularity, such that for all \(0<h\leq h_0\), we have
	\begin{align}\label{eq:full-result}
		\lVert \mathcal{Q}_h(q_1-q_2)\rVert_{H^{-r}(\mathbb{R}^d)}
		\leq C\Phi_r\bigl(h^{1/2}+[\log(2+\delta_\partial^{-1})]^{-1}\bigr).
	\end{align}
\end{theorem}
\begin{proof}
	Let \(Q=\mathcal{Q}_h(q_1-q_2)\) and \(\mu=h^{1/2}+[\log(2+\delta_\partial^{-1})]^{-1}\). We may assume \(\mu\leq\mu_0\) for a sufficiently small fixed \(\mu_0>0\); otherwise, \cref{eq:full-result} follows directly from the a priori bound \(\lVert Q\rVert_{L^2(\mathbb R^d)}\leq2M|\Omega|^{1/2}\) by enlarging \(C\).
	
	We begin by choosing the CGO parameter as
	\begin{align*}
		s=c_1\min\{h^{-1/2},\log(\delta_\partial^{-1})\},
	\end{align*}
	where the logarithm is interpreted as \(+\infty\) at \(\delta_\partial=0\) and \(c_1>0\) is sufficiently small. This choice ensures \(s\asymp\mu^{-1}\) and \(s\geq s_0\). 
	Substituting this into the Fourier estimate \cref{eq:fourier-full}, the terms \(h^2s^3\) and \(hs\) are bounded by \(C\sqrt h\), while the boundary term satisfies
	\begin{align*}
		\delta_\partial e^{Cs}\leq C\delta_\partial^{1/2}\leq C[\log(2+\delta_\partial^{-1})]^{-1}.
	\end{align*}
	Consequently, we obtain a uniform low-frequency bound
	\begin{align*}
		|\widehat Q(\xi)|\leq C\mu,\qquad |\xi|\leq c\mu^{-1}.
	\end{align*}
	
	Next, we split the \(H^{-r}\) norm into low and high frequencies. By Plancherel's theorem, for any \(1\leq R\leq c\mu^{-1}\),
	\begin{align}\label{eq:splitting}
		\lVert Q\rVert_{H^{-r}(\mathbb{R}^d)}^2
		\leq C\mu^2\int_{|\xi|\leq R}(1+|\xi|^2)^{-r}\,d\xi
		+CR^{-2r}\lVert Q\rVert_{L^2(\mathbb{R}^d)}^2.
	\end{align}
	The low-frequency integral is controlled by
	\begin{align*}
		\int_{|\xi|\leq R}(1+|\xi|^2)^{-r}\,d\xi
		\leq C
		\begin{cases}
			R^{d-2r},&0<r<d/2,\\
			\log(2R),&r=d/2,\\
			1,&r>d/2.
		\end{cases}
	\end{align*}
	To balance the two terms in \cref{eq:splitting}, choose \(R=c_2\mu^{-2/d}\) with \(c_2>0\) fixed. Since \(d\geq3\), the condition \(1\leq R\leq c\mu^{-1}\) is satisfied for sufficiently small \(\mu\). Substituting this \(R\) and using the a priori bound for \(\lVert Q\rVert_{L^2(\mathbb R^d)}\), we deduce
	\begin{align*}
		\lVert Q\rVert_{H^{-r}(\mathbb R^d)}^2
		\leq C
		\begin{cases}
			\mu^{4r/d},&0<r<d/2,\\
			\mu^2\log(2+\mu^{-1}),&r=d/2,\\
			\mu^2,&r>d/2.
		\end{cases}
	\end{align*}
	Taking square roots yields exactly \cref{eq:full-result}.
\end{proof}
When the data discrepancy vanishes (\(\delta_\partial=0\)), \cref{eq:full-result} yields the mesh factors \(h^{r/d}\) for \(0<r<d/2\), \(h^{1/2}\sqrt{\log(2+h^{-1/2})}\) for \(r=d/2\), and \(h^{1/2}\) for \(r>d/2\). For the Cartesian realization in \cref{sec:cartesian} with zero boundary nodal potentials, \cref{app:fourier-norms} transfers these bounds to the corresponding negative discrete Fourier norm and identifies the variational boundary normalization. 
The \(\sqrt{h}\) low-frequency scale aligns with that in \cite[Theorem~2.7]{EdG}. While the unweighted low-frequency count in that work only yields the modulus \(\mu^{2r/(2r+d)}\), retaining the Sobolev weight in \cref{eq:splitting} provides the sharper modulus \(\Phi_r(\mu)\). In particular, this yields the bound \(O(\mu)\) for \(r>d/2\), a refinement that also applies to the low-frequency estimate in \cite{EdG}.

\subsection{Two-patch boundary localization}

We now combine the full-data estimate with the boundary continuation theorem to localize the observations. To extend the arguments of the boundary form from a prescribed patch, let $G\subset\Omega$ be a fixed connected open set containing a full interior neighborhood of $\partial\Omega$. Fix a cutoff function \(\eta\in C_c^\infty(\mathbb R^d;[0,1])\) equal to one near \(\partial\Omega\), such that \(\operatorname{supp}\eta\cap\Omega\) lies a positive distance from \(\Omega\setminus G\), and \(\operatorname{supp}\nabla\eta\cap\Omega\Subset A\) for some fixed open set \(A\Subset G\). Assume that the potential difference $p=q_1-q_2$ vanishes in $G$; for lumped mass, we require \(p_i=0\) at every vertex of the elements intersecting \(G\). We refer to these conditions as collar agreement. 
This agreement ensures \(c_h(p;u_h,v_h)=0\) whenever \(\operatorname{supp}v_h\Subset G\). Consequently, the equation for the difference of two solutions is homogeneous along the boundary propagation paths and on the cutoff support.

The following proposition bounds the full flux difference by its restriction to a single boundary patch, while retaining the size of the coefficient difference.
\begin{proposition}\label{prop:homogeneous-continuation}
	Assume \cref{eq:dir-infsup} and collar agreement, and let \(\Gamma\) contain a fixed nonempty patch in the interior of a face. Write $p=q_1-q_2$, $P=\lVert\mathcal Q_hp\rVert_{L^\infty(\mathbb R^d)}$, and $\beta_h=[\Lambda_h(q_1)-\Lambda_h(q_2)]g$ for \(g\in T_h\).
	There exist constants \(C,c,h_0>0\) and \(\vartheta_\Gamma\in(0,1)\) such that for all \(h\leq h_0\),
	\begin{align}\label{eq:homogeneous-continuation}
		\lVert \beta_h\rVert_{T_h^\prime}
		\leq C\lVert \beta_h\rVert_{T_{h,\Gamma}^\prime}^{\vartheta_\Gamma}
		(P\lVert g\rVert_{T_h})^{1-\vartheta_\Gamma}
		+Ce^{-ch^{-4/5}}P\lVert g\rVert_{T_h}.
	\end{align}
	The constants depend only on the fixed data, the collar, the observed patch, and \(C_D\), and are independent of \(p,g,h\) within the prescribed bounds.
\end{proposition}
\begin{proof}
	First, we bound the size of the coefficient difference and the state. Let $u_j=u_h[q_j,g]$ and $w_h=u_1-u_2$. Subtracting the forward equations \cref{eq:forward} for $q_1$ and $q_2$, and applying the inf-sup estimate \cref{eq:dir-infsup}, we obtain
	\begin{align}\label{eq:homogeneous-forward-difference}
		\lVert w_h\rVert_{H^1(\Omega)}+\lVert \beta_h\rVert_{T_h^\prime}
		\leq CP\lVert g\rVert_{T_h}.
	\end{align}
	This follows by testing with zero-trace functions, which gives $a_{q_1,h}(w_h,v_h)=-c_h(p;u_2,v_h)$, and by using the a priori bound $\lVert u_2\rVert_{H^1(\Omega)}\leq C\lVert g\rVert_{T_h}$ from \cref{eq:forward-bound}. To control the flux norm $\lVert\beta_h\rVert_{T_h'}$, we use a minimum-norm trace lifting in the boundary balance identity
	\begin{align}\label{eq:boundary-balance}
		a_{q_1,h}(w_h,v_h)
		=\langle\beta_h,\operatorname{tr}v_h\rangle-c_h(p;u_2,v_h).
	\end{align}
	
	Second, we propagate this estimate to the cutoff transition region. For test functions supported in $G$, the potential term vanishes in \cref{eq:boundary-balance}. Applying the boundary Carleman estimate \cref{thm:boundary-quc} with $F_h=0$, from a sufficiently small ball at the observed face to the set $A$, and using \cref{eq:homogeneous-forward-difference}, we deduce
	\begin{align}\label{eq:collar-transition}
		\lVert w_h\rVert_{H^1(A)}
		\leq C\lVert\beta_h\rVert_{T_{h,\Gamma}^\prime}^{\vartheta_\Gamma}
		(P\lVert g\rVert_{T_h})^{1-\vartheta_\Gamma}
		+Ce^{-ch^{-4/5}}P\lVert g\rVert_{T_h}.
	\end{align}
	
	Finally, we recover the complete flux via duality. For any $k\in T_h$, let $v_h=\overline{u_h[q_1,\overline k]}$. Then $\operatorname{tr}v_h=k$ and $\lVert v_h\rVert_{H^1(\Omega)}\leq C\lVert k\rVert_{T_h}$ by \cref{eq:forward-bound}. By the bilinear symmetry \cref{eq:bilinear-symmetry} and the forward equation, $v_h$ satisfies $a_{q_1,h}(t_h,v_h)=0$ for all $t_h\in V_{h,0}(\Omega)$. 
	Since $\omega_h(\eta)\cap\Omega\subset G$ for small $h$, collar agreement gives $c_h(p;u_2,I_h(\eta v_h))=0$. Observing that $I_h(\eta v_h)$ has trace \(k\) and $I_h(\eta w_h)$ has zero trace, \cref{eq:boundary-balance} yields
	\begin{align}\label{eq:flux-cutoff-commutator}
		\langle\beta_h,k\rangle
		=a_{q_1,h}(w_h,I_h(\eta v_h))-a_{q_1,h}(I_h(\eta w_h),v_h).
	\end{align}
	Since \(\eta\) is real, the element contributions cancel wherever \(\eta\) is constant, leaving contributions only in $A$ for small $h$. By the product estimate \cref{eq:cutoff-product}, the map $z_h\mapsto I_h(\eta z_h)$ is uniformly stable in $H^1(A)$. For lumped mass, the potential terms in \cref{eq:flux-cutoff-commutator} cancel exactly; for consistent mass, they are bounded using $\lVert q_1\rVert_{L^\infty(\Omega)}\leq M$. Elementwise boundedness and the Cauchy--Schwarz inequality thus yield
	\begin{align*}
		|\langle\beta_h,k\rangle|
		\leq C\lVert w_h\rVert_{H^1(A)}\lVert v_h\rVert_{H^1(A)}
		\leq C\lVert w_h\rVert_{H^1(A)}\lVert k\rVert_{T_h}.
	\end{align*}
	Taking the supremum over unit traces \(k\) and using \cref{eq:collar-transition} proves the desired estimate \cref{eq:homogeneous-continuation}.
\end{proof}

For relatively open boundary sets $\Gamma_D,\Gamma_N\subset\partial\Omega$, define
\begin{align}\label{eq:two-patch-data}
	\delta_{\Gamma_D,\Gamma_N}(q_1,q_2)
	=\sup_{\substack{g\in T_{h,\Gamma_D},\ k\in T_{h,\Gamma_N}\\
			\lVert g\rVert_{T_h}=\lVert k\rVert_{T_h}=1}}
	|\langle[\Lambda_h(q_1)-\Lambda_h(q_2)]g,k\rangle|.
\end{align}
Here, the first index specifies the support of the Dirichlet input, and the second specifies the support of the flux observation. In particular, $\delta_{\partial\Omega,\Gamma}=\delta_\Gamma$. The next theorem combines \cref{prop:homogeneous-continuation,thm:full} to bound the coefficient based on independently prescribed input and observation patches.

\begin{theorem}\label{thm:patch}
	Assume the hypotheses of \cref{thm:full}, the uniform estimate \cref{eq:dir-infsup}, and collar agreement. Let $\Gamma_D$ and $\Gamma_N$ each contain a fixed nonempty face-interior patch; these patches may be disjoint. Write $\delta_{\mathrm{loc}}=\delta_{\Gamma_D,\Gamma_N}$ and $P=\lVert\mathcal Q_h(q_1-q_2)\rVert_{L^\infty(\mathbb R^d)}$.
	There exist constants \(C,c,h_0>0\) and \(\vartheta\in(0,1)\) such that for all \(0<h\leq h_0\),
	\begin{align}\label{eq:boundary-cont}
		\delta_\partial\leq C\delta_{\mathrm{loc}}^\vartheta P^{1-\vartheta}
		+Ce^{-ch^{-4/5}}P.
	\end{align}
	Moreover, for every $r>0$, we have
	\begin{align}\label{eq:patch-result}
		\lVert \mathcal Q_h(q_1-q_2)\rVert_{H^{-r}(\mathbb{R}^d)}
		\leq C\Phi_r\bigl(h^{1/2}+[\log(2+\delta_{\mathrm{loc}}^{-1})]^{-1}\bigr).
	\end{align}
	These constants depend only on the fixed geometry, collar, patches, potential bound, mesh regularity, \(C_D\), and the exponent \(r\) in \cref{eq:patch-result}.
\end{theorem}

\begin{proof}
	\emph{Step 1: Extend the flux observation to the full boundary.}
	For each unit trace $g\in T_{h,\Gamma_D}$, applying \cref{prop:homogeneous-continuation} with observed patch $\Gamma_N$ bounds the restricted flux difference by $\delta_{\mathrm{loc}}$. Taking the supremum over these Dirichlet traces yields
	\begin{align}\label{eq:first-boundary-continuation}
		\delta_{\Gamma_D,\partial\Omega}
		\leq C\delta_{\mathrm{loc}}^{\vartheta_N}P^{1-\vartheta_N}+Ce^{-ch^{-4/5}}P,
	\end{align}
	where $\vartheta_N$ is the exponent supplied by \cref{prop:homogeneous-continuation} for $\Gamma_N$.
	
	\emph{Step 2: Interchange input and observation by reciprocity.}
	Using $\overline{u_h[q,\overline k]}$ as a lifting of \(k\) in \cref{eq:dtn}, the bilinear symmetry \cref{eq:bilinear-symmetry} yields the reciprocity relation
	\begin{align}\label{eq:boundary-reciprocity}
		\langle\Lambda_h(q)g,k\rangle
		=\langle\Lambda_h(q)\overline k,\overline g\rangle.
	\end{align}
	Applying this to the difference of the two boundary maps, and noting that complex conjugation preserves trace norms and supports, we obtain
	\begin{align*}
		\delta_{\Gamma_D,\partial\Omega}=\delta_{\partial\Omega,\Gamma_D}.
	\end{align*}
	Consequently, \cref{eq:first-boundary-continuation} also controls the flux observed on $\Gamma_D$ for arbitrary unit Dirichlet traces on the full boundary.
	
	\emph{Step 3: Extend the Dirichlet input to the full boundary.}
	Applying \cref{prop:homogeneous-continuation} with observed patch $\Gamma_D$ and an arbitrary unit trace $g\in T_h$ (with exponent $\vartheta_D$), and taking the supremum over \(g\), we use Step 2 and \cref{eq:first-boundary-continuation} to deduce
	\begin{align*}
		\delta_\partial
		\leq C\delta_{\partial\Omega,\Gamma_D}^{\vartheta_D}P^{1-\vartheta_D}
		+Ce^{-ch^{-4/5}}P
		\leq C\delta_{\mathrm{loc}}^{\vartheta_D\vartheta_N}P^{1-\vartheta_D\vartheta_N}
		+Ce^{-c'h^{-4/5}}P.
	\end{align*}
	Here, the second inequality follows from the subadditivity $(a+b)^{\vartheta_D}\leq a^{\vartheta_D}+b^{\vartheta_D}$ for $a,b\geq0$, and the constant $c'>0$ is reduced to absorb the exponential remainder from the first continuation. This proves \cref{eq:boundary-cont} with $\vartheta=\vartheta_D\vartheta_N$.
	
	\emph{Step 4: Apply full-data stability.}
	Using the a priori bound $P\leq2M$ in \cref{eq:boundary-cont}, the right-hand side is controlled by a constant multiple of the larger of $\delta_{\mathrm{loc}}^\vartheta$ and $e^{-ch^{-4/5}}$. Taking the reciprocal logarithm gives
	\begin{align}\label{eq:patch-log-modulus}
		[\log(2+\delta_\partial^{-1})]^{-1}
		\leq C\bigl([\log(2+\delta_{\mathrm{loc}}^{-1})]^{-1}+h^{4/5}\bigr),
	\end{align}
	which naturally includes the case $\delta_{\mathrm{loc}}=0$. Since \(h^{4/5}\leq h^{1/2}\) for \(h\leq1\), applying the full-data estimate \cref{thm:full} and using the monotonicity of $\Phi_r$ and the subadditivity $\Phi_r(Ct)\leq C\Phi_r(t)$ yields the final result \cref{eq:patch-result}. The case of large data is trivially covered by the a priori \(L^2\) bound.
\end{proof}


\subsection{A finite-dimensional consequence}


Related Lipschitz stability results for finite-dimensional coefficient models in the continuous setting include \cite{AS2019,RSincich2022}. For the finite-dimensional consequence below, we retain the factor $P=\lVert\mathcal Q_h(q_1-q_2)\rVert_{L^\infty(\mathbb{R}^d)}$ from \cref{eq:fourier-full,eq:boundary-cont}.
\begin{corollary}\label{cor:homogeneous-inverse}
	Let $r>0$ and $P=\lVert\mathcal Q_h(q_1-q_2)\rVert_{L^\infty(\mathbb R^d)}>0$. Assume either the hypotheses of \cref{thm:full} with $\delta=\delta_\partial$, or those of \cref{thm:patch} with $\delta=\delta_{\mathrm{loc}}$. Then there exist constants \(C,h_0>0\) such that for all \(0<h\leq h_0\),
	\begin{align}\label{eq:homogeneous-inverse}
		\lVert\mathcal Q_h(q_1-q_2)\rVert_{H^{-r}(\mathbb R^d)}
		\leq CP\Phi_r\bigl(h^{1/2}+[\log(2+P/\delta)]^{-1}\bigr).
	\end{align}
	The constants share the dependence specified in the corresponding theorem and are independent of \(P\). The reciprocal logarithm is interpreted as zero when \(\delta=0\).
\end{corollary}
\begin{proof}
	For the full-data estimate, we normalize by the amplitude \(P\). Dividing \cref{eq:fourier-full} by \(P\) and repeating the frequency-splitting argument from the proof of \cref{thm:full} for the normalized coefficient \(\mathcal{Q}_h(q_1-q_2)/P\) (whose \(L^2\) norm is uniformly bounded) yields the estimate with normalized data \(\delta_\partial/P\).
	
	For the local-data estimate, dividing the boundary continuation bound \cref{eq:boundary-cont} by \(P\) gives
	\begin{align*}
		\delta_\partial/P\leq C(\delta_{\mathrm{loc}}/P)^\vartheta+Ce^{-ch^{-4/5}}.
	\end{align*}
	Applying the logarithmic comparison used in \cref{eq:patch-log-modulus} to these normalized data then yields the desired local estimate. Throughout this normalization procedure, the a priori \(L^\infty\) bound \(M\) for the potentials remains unchanged.
\end{proof}

For coefficient differences restricted to a fixed finite-dimensional space, the mesh remainder in \cref{eq:homogeneous-inverse} can be absorbed under either mass realization. Let $\pi_h$ denote the coefficient projection. For consistent mass and $p\in L^\infty(\Omega;\mathbb C)$, we set $\pi_hp=p$. For lumped mass, we use the full nodal partition $\{D_i\}$ of $\Omega$ from the coefficient realization (including the control volumes at boundary vertices) and define
\begin{align*}
	(\pi_hp)_i=\frac1{m_i}\int_{D_i}p(x)\,dx.
\end{align*}

\begin{corollary}\label{cor:finite-dimensional}
	Fix a finite-dimensional subspace $W\subset L^\infty(\Omega;\mathbb C)$ independent of \(h\), and let $p\in W$ with $q_1-q_2=\pi_hp$. Let $\delta=\delta_\partial$ under the hypotheses of \cref{thm:full}, and $\delta=\delta_{\mathrm{loc}}$ under those of \cref{thm:patch}. Then there exist constants \(C,h_W>0\), depending only on \(W\) and the fixed data in the corresponding theorem, such that for all \(0<h\leq h_W\),
	\begin{align}\label{eq:finite-dimensional-stability}
		\lVert p\rVert_{L^\infty(\Omega)}\leq C\delta.
	\end{align}
\end{corollary}
\begin{proof}
	We set \(r=d\) and let $p^0$ denote the zero extension of \(p\). The case \(W=\{0\}\) is immediate, so assume \(W\) is nontrivial. By finite-dimensional norm equivalence on \(W\), we have
	\begin{align*}
		\lVert p\rVert_{L^\infty(\Omega)}\leq C\lVert p^0\rVert_{H^{-d}(\mathbb R^d)}.
	\end{align*}
	
	Next, we estimate the discrepancy between the projected coefficient and its continuous counterpart in the \(H^{-d}\) norm. For the lumped mass case, let $\varphi\in H^d(\mathbb R^d)$. The cancellation of cell averages and the bound $\operatorname{diam}D_i\leq Ch$ yield
	\begin{align*}
		\Bigl|\int_{\mathbb R^d}(\mathcal Q_h\pi_hp-p^0)\overline\varphi\,dx\Bigr|
		&\leq\sum_i\frac1{m_i}\int_{D_i}\int_{D_i}|p(y)|\,|\varphi(x)-\varphi(y)|\,dy\,dx\\
		&\leq Ch\lVert p\rVert_{L^1(\Omega)}\lVert\nabla\varphi\rVert_{L^\infty(\mathbb R^d)}
		\leq Ch\lVert p\rVert_{L^\infty(\Omega)}\lVert\varphi\rVert_{H^d(\mathbb R^d)},
	\end{align*}
	where the last inequality relies on the Sobolev embedding $H^d(\mathbb R^d)\hookrightarrow W^{1,\infty}(\mathbb R^d)$, which is valid since $d\geq3$. Taking the supremum over unit \(\varphi\) gives
	\begin{align*}
		\lVert\mathcal Q_h\pi_hp-p^0\rVert_{H^{-d}(\mathbb R^d)}\leq Ch\lVert p\rVert_{L^\infty(\Omega)}.
	\end{align*}
	(For consistent mass, the left-hand side is identically zero.) For sufficiently small \(h\), this estimate and the finite-dimensional norm equivalence imply
	\begin{align}\label{eq:coefficient-uniform-injectivity}
		P_h\leq\lVert p\rVert_{L^\infty(\Omega)}\leq C\lVert\mathcal Q_h\pi_hp\rVert_{H^{-d}(\mathbb R^d)}\leq CP_h,
	\end{align}
	where \(P_h=\lVert\mathcal Q_h\pi_hp\rVert_{L^\infty(\mathbb R^d)}\), and the last inequality follows from \(\lVert\mathcal Q_h\pi_hp\rVert_{L^2(\mathbb R^d)}\leq|\Omega|^{1/2}P_h\).
	
	Finally, we establish the stability bound. If \(P_h=0\), then \(p=0\) by \cref{eq:coefficient-uniform-injectivity}, and the assertion holds trivially. Otherwise, applying \cref{eq:homogeneous-inverse} with \(\Phi_d(t)=t\), dividing by \(P_h>0\), and using the lower bound in \cref{eq:coefficient-uniform-injectivity} yields
	\begin{align*}
		1\leq C\bigl(h^{1/2}+[\log(2+P_h/\delta)]^{-1}\bigr).
	\end{align*}
	Choose \(h_W\leq h_0\) small enough so that \(Ch_W^{1/2}\leq1/2\). This forces \([\log(2+P_h/\delta)]^{-1}\geq1/(2C)\), which implies \(\delta\geq cP_h\). Combining this with \cref{eq:coefficient-uniform-injectivity} yields \(\lVert p\rVert_{L^\infty(\Omega)}\leq C\delta\), completing the proof of \cref{eq:finite-dimensional-stability} (including the uniqueness conclusion when \(\delta=0\)).
\end{proof}

Consequently, every fixed finite-dimensional coefficient space admits uniqueness and Lipschitz stability without a mesh remainder for all sufficiently small \(h\). \Cref{cor:finite-dimensional} applies, for example, to piecewise constant or piecewise polynomial potentials on a fixed partition, even as the forward mesh is refined independently. For lumped mass, control-volume averages circumvent the need to assign point values to \(L^\infty\) coefficients.

We note that the forward solvability and (for local data) collar agreement assumptions are imposed on \(q_1,q_2\), which differ by \(\pi_hp\). For local data, a sufficient condition for collar agreement after averaging is that \(p\) vanishes on a fixed open neighborhood \(G_+\) of the propagation collar \(G\), with \(\operatorname{dist}(G,\Omega\setminus G_+)>0\). Consequently, for sufficiently small \(h\), the control volumes associated with all vertices of elements meeting \(G\) lie in \(G_+\).

\appendix

\section{A continuous radial Carleman estimate}\label[appendix]{app:radial}
We prove the continuous estimate used in \cref{thm:radial-discrete}. Let \(D\) be a fixed convex polytope, \(x_*\in D\), and \(B_*\Subset D\) a fixed ball centered at \(x_*\). For the radial phase in \cref{eq:radial-weight}, set \(\psi=-|x-x_*|^2\) so that $\rho=e^{\lambda\psi}$ and $\theta=e^{s\rho}$.

The proof proceeds in three steps. First, we compute the commutator of the conjugated Laplacian with its symmetric and skew-symmetric parts to control the weighted \(L^2\) norm. Second, we recover the positive gradient term via an energy pairing, choosing the phase parameter \(\lambda\) large enough to absorb the lower-order terms outside the observation ball. Third, we extend the result from smooth convex domains to the given polytope by a standard approximation argument.

\begin{lemma}\label{lem:radial-cont}
	For any fixed, sufficiently large $\lambda$, there exist constants \(s_0, C > 0\), depending only on \(D\), \(B_*\), and the fixed phase, such that for all \(s\geq s_0\) and \(v\in H^2(D)\cap H_0^1(D)\),
	\begin{align}\label{eq:app-radial}
		s\lVert \theta\nabla v\rVert_{L^2(D)}^2+s^3\lVert \theta v\rVert_{L^2(D)}^2
		\leq C\bigl(\lVert \theta\Delta v\rVert_{L^2(D)}^2+s^3\lVert \theta v\rVert_{L^2(B_*)}^2\bigr).
	\end{align}
\end{lemma}

\begin{proof}
	\emph{The conjugated identity.} 
	Assume first that \(D\) is a smooth convex domain and \(v\) is a smooth zero-trace function. Set \(z=\theta v\) and write the conjugated Laplacian as \(P_s=\theta(-\Delta)\theta^{-1}=\mathsf{A}+\mathsf{B}\), where \(\mathsf{A}=-\Delta-s^2|\nabla\rho|^2\) and \(\mathsf{B}=2s\nabla\rho\cdot\nabla+s\Delta\rho\).
	Repeated integration by parts yields the commutator identity
	\begin{align}\label{eq:radial-commutator}
		2\operatorname{Re}(\mathsf{A}z,\mathsf{B}z)
		&=4s\int_D D^2\rho(\nabla z,\nabla\bar z)\,dx
		+4s^3\int_D D^2\rho(\nabla\rho,\nabla\rho)|z|^2\,dx\notag\\
		&\quad-s\int_D\Delta^2\rho\,|z|^2\,dx
		-2s\int_{\partial D}\partial_\nu\rho\,|\partial_\nu z|^2\,dS.
	\end{align}
	Here $\nu$ is the outward unit normal. The boundary contribution is nonnegative, since convexity of \(D\) with respect to \(x_*\in D\) gives $\partial_\nu\rho=-2\lambda\rho(x-x_*)\cdot\nu\leq0$.
	Using the explicit Hessian $D^2\rho=\lambda\rho(-2\mathrm{Id}+\lambda\nabla\psi\otimes\nabla\psi)$ and discarding nonnegative terms in \cref{eq:radial-commutator} (specifically, $\lambda(\nabla\psi\cdot\nabla z)^2$), we obtain the lower bound
	\begin{align}\label{eq:radial-first}
		\lVert P_sz\rVert_{L^2(D)}^2
		&\geq-8s\lambda\int_D\rho|\nabla z|^2\,dx
		+4s^3\lambda^3\int_D\rho^3
		\bigl(\lambda|\nabla\psi|^4-2|\nabla\psi|^2\bigr)|z|^2\,dx
		-C(\lambda)s\lVert z\rVert_{L^2(D)}^2.
	\end{align}
	
	\emph{Recovery of coercive terms.} 
	The bound \cref{eq:radial-first} has a positive zeroth-order term but a negative gradient term. To close the estimate, we pair $P_sz$ with $\lambda\rho z$. This yields the energy identity
	\begin{align}\label{eq:radial-energy}
		s\lambda\int_D\rho|\nabla z|^2\,dx
		&\leq\lVert P_sz\rVert_{L^2(D)}^2
		+s^3\lambda^3\int_D\rho^3|\nabla\psi|^2|z|^2\,dx
		+C(\lambda)s^2\lVert z\rVert_{L^2(D)}^2,
	\end{align}
	where we used Young's inequality to absorb the linear term $s\lambda\operatorname{Re}(P_sz,\rho z)$.
	Substituting \cref{eq:radial-energy} into \cref{eq:radial-first} gives
	\begin{align}\label{eq:radial-positive}
		9\lVert P_sz\rVert_{L^2(D)}^2\geq
		s^3\lambda^3\int_D\rho^3\bigl(4\lambda|\nabla\psi|^4-16|\nabla\psi|^2\bigr)|z|^2\,dx
		-C(\lambda)s^2\lVert z\rVert_{L^2(D)}^2.
	\end{align}
	We now fix $\lambda$ sufficiently large so that $\lambda|\nabla\psi|^2\geq8$ outside the observation ball \(B_*\). Then the coefficient of $s^3|z|^2$ in \cref{eq:radial-positive} has a strictly positive lower bound on $D\setminus B_*$. For large \(s\), the negative term $-C(\lambda)s^2\lVert z\rVert_{L^2(D)}^2$ can be absorbed into this positive exterior contribution, leaving only an interior remainder $Cs^3\lVert z\rVert_{L^2(B_*)}^2$. This controls $s^3\lVert z\rVert_{L^2(D)}^2$, and \cref{eq:radial-energy} subsequently controls $s\lVert\nabla z\rVert_{L^2(D)}^2$. Finally, the identity $\theta\nabla v=\nabla z-s\nabla\rho\,z$ proves \cref{eq:app-radial} for smooth convex domains.
	
	\emph{Passage to a convex polytope.} 
	It remains to extend the estimate to the given polytope \(D\). Choose smooth convex domains $D_j\uparrow D$ containing $B_*$ for large $j$. Given $v\in H^2(D)\cap H_0^1(D)$, set $f=(-\Delta+1)v$ and solve $(-\Delta+1)v_j=f|_{D_j}$ with zero boundary trace. Let $\widetilde v_j$ denote the zero extension to \(D\). The variational equations bound $\widetilde v_j$ in \(H_0^1(D)\). By testing against compactly supported functions, every weak limit coincides with \(v\). The energy identities
	\begin{align*}
		\lVert\widetilde v_j\rVert_{H^1(D)}^2
		=(f,\widetilde v_j)_{L^2(D)}
		\longrightarrow(f,v)_{L^2(D)}
		=\lVert v\rVert_{H^1(D)}^2
	\end{align*}
	thus yield strong $H_0^1(D)$ convergence. Furthermore, the zero-extended Laplacians satisfy
	\begin{align*}
		\widetilde{\Delta v_j}
		=\widetilde v_j-\mathbf{1}_{D_j}f
		\longrightarrow v-f=\Delta v\quad\text{in }L^2(D).
	\end{align*}
	The constants in \cref{eq:app-radial} on $D_j$ depend only on fixed phase derivatives and the uniform sign of the boundary term in \cref{eq:radial-commutator}, making them uniform in \(j\). By density, \cref{eq:app-radial} extends to $H^2(D_j)\cap H_0^1(D_j)$. Passing to the limit for each fixed \(s\) then proves \cref{eq:app-radial} on \(D\).
\end{proof}

\section{Cartesian Fourier norms and boundary normalization}\label[appendix]{app:fourier-norms}

Standard Cartesian finite differences admit a lumped-mass finite-element realization, as established in \cref{sec:cartesian}. However, our inverse stability results are stated in the continuous negative Sobolev norm \(H^{-r}(\mathbb{R}^d)\), whereas discrete boundary data naturally live in grid-based norms. This appendix serves two purposes: first, to prove that the continuous norm controls the corresponding discrete Fourier norm; and second, to clarify the normalization of the variational boundary pairing on the Cartesian grid, showing that it matches the standard discrete Dirichlet-to-Neumann operator norm.

\subsection*{Comparison of continuous and discrete negative norms}

Let \(h=1/N\) and let \(q_i\) be nodal values on the unit cube, vanishing at the boundary nodes. We regard these values as an \(N\)-periodic array. Let \(Q\) be the piecewise constant reconstruction on centered interior dual cubes, extended by zero on the remaining boundary strip and outside the cube. Let \(Q_{\mathrm{per}}\) be its periodization on \(\mathbb{T}^d=\mathbb{R}^d/\mathbb{Z}^d\). Define the discrete Fourier coefficients by
\begin{align*}
	F_hq(k)=h^d\sum_i q_i e^{-2\pi i k\cdot x_i}.
\end{align*}
Let \(\mathcal{K}_N\) be a centered set of representatives modulo \(N\mathbb{Z}^d\). The corresponding discrete negative norm is defined by
\begin{align*}
	|q|_{H_h^{-r}}^2=\sum_{k\in\mathcal{K}_N}(1+|k|^2)^{-r}|F_hq(k)|^2.
\end{align*}
The following lemma compares this discrete norm with its continuous counterpart, showing that the \(H^{-r}(\mathbb{R}^d)\) norm is stronger.
\begin{lemma}\label{lem:fourier-norm-compare}
	For every \(r>0\), there exists a constant \(C>0\), depending only on \(d\) and \(r\), such that
	\begin{align}\label{eq:negative-norm-compare}
		|q|_{H_h^{-r}}\leq C\lVert Q\rVert_{H^{-r}(\mathbb{R}^d)}.
	\end{align}
	The same bound holds for noncentered frequency representatives.
\end{lemma}
\begin{proof}
	We first compare the discrete norm with the periodic norm. Let 
	\(\widehat{Q_{\mathrm{per}}}(k)=\int_{[0,1]^d}Q_{\mathrm{per}}(x)e^{-2\pi i k\cdot x}\,dx\) 
	be the periodic Fourier transform, and set \(\operatorname{sinc}(t)=\sin(t)/t\) (with \(\operatorname{sinc}(0)=1\)). Integrating over the dual cubes yields
	\begin{align*}
		\widehat{Q_{\mathrm{per}}}(k+N\ell)=F_hq(k)
		\prod_{j=1}^d\operatorname{sinc}\bigl(\pi(k_j/N+\ell_j)\bigr),\qquad \ell\in\mathbb{Z}^d.
	\end{align*}
	In the principal cell, the modulus of the sinc product is at least \((2/\pi)^d\). For centered representatives, we have \(|k_j+N\ell_j|\geq|k_j|\) in each coordinate, and the Parseval identity for the piecewise constant cell basis gives
	\begin{align*}
		\sum_{\ell\in\mathbb{Z}^d}\prod_{j=1}^d
		\operatorname{sinc}^2\bigl(\pi(k_j/N+\ell_j)\bigr)=1.
	\end{align*}
	Grouping these aliasing terms thus establishes the equivalence
	\begin{align*}
		\lVert Q_{\mathrm{per}}\rVert_{H^{-r}(\mathbb{T}^d)}\asymp |q|_{H_h^{-r}}.
	\end{align*}
	
	Next, we relate the periodic norm to the whole-space norm. Multiplying a periodic test function by a fixed smooth cutoff equal to one on a neighborhood of \([0,1]^d\), local multiplication and a finite covering argument show that its whole-space \(H^r\) norm is controlled by the periodic \(H^r\) norm. By duality, we obtain
	\begin{align*}
		\lVert Q_{\mathrm{per}}\rVert_{H^{-r}(\mathbb{T}^d)}\leq C\lVert Q\rVert_{H^{-r}(\mathbb{R}^d)}.
	\end{align*}
	For noncentered representatives, the frequency weight does not exceed that of the centered alias; any duplicated endpoints merely alter the bound by a dimension-dependent constant. Combining these comparisons yields \cref{eq:negative-norm-compare}.
\end{proof}

\subsection*{Normalization of the boundary pairing}

We finally verify that the variational boundary flux in our discrete Calder\'on problem matches the standard normalization of the Cartesian finite-difference Dirichlet-to-Neumann map. Let \(\Omega_h\) and \(\partial\Omega_h\) denote the interior and boundary grid nodes, and set \(\Gamma_h=\Gamma\cap\partial\Omega_h\). The boundary flux pairing uses the boundary measure \(h^{d-1}\sum_{\partial\Omega_h}\), while the interior mass uses \(h^d\sum_{\Omega_h}\). This is precisely the normalization in the variational Green identity for the Cartesian Laplacian \cref{eq:cartesian-laplacian}; the factor \(1/h\) in \cite[(2.16)]{EdG} converts its boundary sum to this measure. Note that any boundary-to-boundary stiffness block independent of \(q\) cancels in a difference of Dirichlet-to-Neumann maps.

Equip boundary nodal values with the norm \(\lVert g\rVert_{L_h^2(\Gamma_h)}^2=h^{d-1}\sum_{x_i\in\Gamma_h}|g_i|^2\). Identifying each variational boundary functional with its Riesz representative in \(L_h^2(\partial\Omega_h)\), uniform trace boundedness yields
\begin{align*}
	\delta_\Gamma
	\leq C\lVert \Lambda_h(q_1)-\Lambda_h(q_2)\rVert_{L_h^2(\partial\Omega_h)\to L_h^2(\Gamma_h)}.
\end{align*}
This shows that the comparison between the \(L^2\) data norm and the natural trace norm is uniform in \(h\). Consequently, when \(\Gamma=\partial\Omega\) and the discrepancy is measured in this \(L_h^2\) operator norm, the estimate \cref{eq:full-result} holds in \(H_h^{-r}\) with the same modulus \(\Phi_r\).

\bibliographystyle{siam-paper-no}
\bibliography{references}

@article{Boyer2010,
  author = {Boyer, Franck and Hubert, Florence and {Le Rousseau}, J{\'e}r{\^o}me},
  title = {Discrete {Carleman} estimates for elliptic operators and uniform controllability of semi-discretized parabolic equations},
  journal = {J. Math. Pures Appl.},
  volume = {93},
  number = {3},
  pages = {240--276},
  year = {2010},
  doi = {10.1016/j.matpur.2009.11.003}
}

@book{ern2004,
  author    = {Ern, Alexandre and Guermond, Jean-Luc},
  title     = {Theory and Practice of Finite Elements},
  series    = {Applied Mathematical Sciences},
  volume    = {159},
  publisher = {Springer},
  address   = {New York},
  year      = {2004},
  doi       = {10.1007/978-1-4757-4355-5}
}

@article{EdG,
  author = {Ervedoza, Sylvain and de Gournay, Fr{\'e}d{\'e}ric},
  title = {Uniform stability estimates for the discrete {Calder{\'o}n} problems},
  journal = {Inverse Problems},
  volume = {27},
  number = {12},
  pages = {Paper No.~125012},
  year = {2011},
  doi = {10.1088/0266-5611/27/12/125012}
}

@article{FBRRS,
  author = {Fern{\'a}ndez-Bertolin, Aingeru and Roncal, Luz and R{\"u}land, Angkana and Stan, Diana},
  title = {Discrete {Carleman} estimates and three balls inequalities},
  journal = {Calc. Var. Partial Differential Equations},
  volume = {60},
  pages = {Paper No.~239},
  year = {2021},
  doi = {10.1007/s00526-021-02098-z}
}

@book{Grisvard2011,
  author = {Grisvard, Pierre},
  title = {Elliptic Problems in Nonsmooth Domains},
  series = {Classics in Applied Mathematics},
  volume = {69},
  publisher = {Society for Industrial and Applied Mathematics},
  address = {Philadelphia},
  year = {2011},
  doi = {10.1137/1.9781611972030}
}

@article{Lecaros2023a,
  author = {Lecaros, Rodrigo and Ortega, Jaime H. and P{\'e}rez, Ariel and {De Teresa}, Luz},
  title = {Discrete {Calder{\'o}n} problem with partial data},
  journal = {Inverse Problems},
  volume = {39},
  number = {3},
  pages = {Paper No.~035001},
  year = {2023},
  doi = {10.1088/1361-6420/acb0f8}
}

@article{LR2012,
  author = {{Le Rousseau}, J{\'e}r{\^o}me and Lebeau, Gilles},
  title = {On {Carleman} estimates for elliptic and parabolic operators. {Applications} to unique continuation and control of parabolic equations},
  journal = {ESAIM Control Optim. Calc. Var.},
  volume = {18},
  number = {3},
  pages = {712--747},
  year = {2012},
  doi = {10.1051/cocv/2011168}
}

@article{SZ,
  author = {Scott, L. Ridgway and Zhang, Shangyou},
  title = {Finite element interpolation of nonsmooth functions satisfying boundary conditions},
  journal = {Math. Comp.},
  volume = {54},
  number = {190},
  pages = {483--493},
  year = {1990},
  doi = {10.1090/S0025-5718-1990-1011446-7}
}

@article{ZY,
  author = {Zhao, Xiaomeng and Yuan, Ganghua},
  title = {Stability estimate for the discrete {Calder{\'o}n} problem with partial data},
  journal = {ESAIM Control Optim. Calc. Var.},
  volume = {32},
  pages = {Paper No.~41},
  year = {2026},
  doi = {10.1051/cocv/2026023}
}

@article{RS2019,
  author = {R{\"u}land, Angkana and Salo, Mikko},
  title = {Quantitative {Runge} approximation and inverse problems},
  journal = {Int. Math. Res. Not. IMRN},
  year = {2019},
  number = {20},
  pages = {6216--6234},
  doi = {10.1093/imrn/rnx301}
}

@article{AS2019,
  author = {Alberti, Giovanni S. and Santacesaria, Matteo},
  title = {{Calder{\'o}n}'s inverse problem with a finite number of measurements},
  journal = {Forum Math. Sigma},
  volume = {7},
  pages = {Paper No.~e35},
  year = {2019},
  doi = {10.1017/fms.2019.31}
}

@article{RSincich2022,
  author = {R{\"u}land, Angkana and Sincich, Eva},
  title = {On {Runge} approximation and {Lipschitz} stability for a finite-dimensional {Schr{\"o}dinger} inverse problem},
  journal = {Appl. Anal.},
  volume = {101},
  number = {10},
  pages = {3655--3666},
  year = {2022},
  doi = {10.1080/00036811.2020.1738403}
}

@misc{ParabolicTransfer,
  author = {L{\"u}, Qi and Wang, Yu},
  title = {From continuous to fully discrete {Carleman} estimates: A transfer principle for parabolic schemes},
  year = {2026},
  howpublished = {arXiv:2610.02983 }
}

@article{ARRV2009,
  author = {Alessandrini, Giovanni and Rondi, Luca and Rosset, Edi and Vessella, Sergio},
  title = {The stability for the {Cauchy} problem for elliptic equations},
  journal = {Inverse Problems},
  volume = {25},
  number = {12},
  pages = {Paper No.~123004},
  year = {2009},
  doi = {10.1088/0266-5611/25/12/123004}
}

@article{Boyer2010a,
  author = {Boyer, Franck and Hubert, Florence and {Le Rousseau}, J{\'e}r{\^o}me},
  title = {Discrete {Carleman} estimates for elliptic operators in arbitrary dimension and applications},
  journal = {SIAM J. Control Optim.},
  volume = {48},
  number = {8},
  pages = {5357--5397},
  year = {2010},
  doi = {10.1137/100784278}
}

@article{GM2014,
  author = {Guadie, Maru and Malinnikova, Eugenia},
  title = {On three balls theorem for discrete harmonic functions},
  journal = {Comput. Methods Funct. Theory},
  volume = {14},
  pages = {721--734},
  year = {2014},
  doi = {10.1007/s40315-014-0076-9}
}

@article{LM2015,
  author = {Lippner, Gabor and Mangoubi, Dan},
  title = {Harmonic functions on the lattice: {Absolute} monotonicity and propagation of smallness},
  journal = {Duke Math. J.},
  volume = {164},
  number = {13},
  pages = {2577--2595},
  year = {2015},
  doi = {10.1215/00127094-3164790}
}

@incollection{Zuazua2007,
  author = {Zuazua, Enrique},
  title = {Controllability and observability of partial differential equations: some results and open problems},
  booktitle = {Handbook of Differential Equations: Evolutionary Equations. {V}ol. {III}},
  publisher = {Elsevier/North-Holland, Amsterdam},
  pages = {527--621},
  year = {2007},
  doi = {10.1016/S1874-5717(07)80010-7}
}

@article{MN2003,
  author = {Makridakis, Charalambos and Nochetto, Ricardo H.},
  title = {Elliptic reconstruction and a posteriori error estimates for parabolic problems},
  journal = {SIAM J. Numer. Anal.},
  volume = {41},
  number = {4},
  pages = {1585--1594},
  year = {2003},
  doi = {10.1137/S0036142902406314}
}

@article{NS1974,
  author = {Nitsche, Joachim A. and Schatz, Alfred H.},
  title = {Interior estimates for {Ritz--Galerkin} methods},
  journal = {Math. Comp.},
  volume = {28},
  pages = {937--958},
  year = {1974},
  doi = {10.1090/S0025-5718-1974-0373325-9}
}

@article{AK2012,
  author = {Alessandrini, Giovanni and Kim, Kyoungsun},
  title = {Single-logarithmic stability for the {Calder{\'o}n} problem with local data},
  journal = {J. Inverse Ill-Posed Probl.},
  volume = {20},
  number = {4},
  pages = {389--400},
  year = {2012},
  doi = {10.1515/jip-2012-0014}
}

@article{SU1987,
  author = {Sylvester, John and Uhlmann, Gunther},
  title = {A global uniqueness theorem for an inverse boundary value problem},
  journal = {Ann. of Math. (2)},
  volume = {125},
  number = {1},
  pages = {153--169},
  year = {1987},
  doi = {10.2307/1971291}
}

@article{Alessandrini1988,
  author = {Alessandrini, Giovanni},
  title = {Stable determination of conductivity by boundary measurements},
  journal = {Appl. Anal.},
  volume = {27},
  number = {1--3},
  pages = {153--172},
  year = {1988},
  doi = {10.1080/00036818808839730}
}

@article{KSU2007,
  author = {Kenig, Carlos E. and Sj{\"o}strand, Johannes and Uhlmann, Gunther},
  title = {The {Calder{\'o}n} problem with partial data},
  journal = {Ann. of Math. (2)},
  volume = {165},
  number = {2},
  pages = {567--591},
  year = {2007},
  doi = {10.4007/annals.2007.165.567}
}
\end{document}